\documentclass[12pt]{amsart}
\usepackage[utf8]{inputenc}
\usepackage{amsfonts}
\usepackage{latexsym}
\usepackage{amssymb}
\usepackage{amsmath}
\usepackage[dvipsnames]{xcolor}
\usepackage[foot]{amsaddr}
\usepackage{tikz}
\usepackage{enumerate}
\usepackage{mathrsfs}
\usepackage{todonotes}
\usepackage[left=1.4cm, top=2.5cm,bottom=2.5cm,right=1.4cm]{geometry}
\usepackage{bbm}

\usepackage[dvipsnames]{xcolor}
\usepackage{forest}
\forestset{
  default preamble={
   for tree={circle, draw, 
            minimum size=2.1em, 
            inner sep=1pt} 
  }
}
 \definecolor{col1}{rgb}{0.7, 0.86, 0.96}
\definecolor{col2}{rgb}{0.95, 0.95, 0.65}
\definecolor{col3}{rgb}{0.8, 0.9, 0.5}
\definecolor{col4}{rgb}{0.91,0.94, 0.53}
\definecolor{col5}{rgb}{0.98,0.99,0.6}

\definecolor{texted}{rgb}{0.1, 0.6, 0.3}
\usetikzlibrary{positioning}

\usepackage{times}

\usetikzlibrary{backgrounds}

\numberwithin{equation}{section}

\newtheorem{thm}{Theorem}[section]
\newtheorem{cor}[thm]{Corollary}
\newtheorem{lem}[thm]{Lemma}
\newtheorem{prop}[thm]{Proposition}

\newtheorem{defn}[thm]{Definition}

\newtheorem{notation}[thm]{Notation}

\newcommand{\la}{\lambda}
\newcommand{\tO}{\tilde{\Omega}}

\newcommand{\BeG}{\begin{align*}}
\newcommand{\EnD}{\end{align*}}
\newcommand{\bae}{\begin{align}}
\newcommand{\eae}{\end{align}}
\newcommand{\Q}{\mathbb{Q}}
\renewcommand{\P}{\mathbb{P}}

\newcommand{\F}{\mathcal{F}}

\newcommand{\ind}{\mathbbm{1}}

\newcommand{\bp}{\begin{proof}}
\newcommand{\ep}{\end{proof}}
\def\bal#1\eal{\begin{align*}#1\end{align*}}

\newcommand{\Zc}{\mathcal{Z}}

\renewcommand{\S}{\mathcal{S}}

\newcommand{\tF}{\tilde{\mathcal{F}}}

\DeclareMathOperator{\Leb}{Leb}

\title[]{The coalescent structure of uniform and Poisson samples from multitype branching processes}

\author[S.G.G. Johnston]{Samuel G. G. Johnston}
\address{Samuel G. G. Johnston: Department of Mathematical Sciences, University of Bath, United Kingdom.}
\email{sgj22@bath.ac.uk}

\author[A. Lambert]{Amaury Lambert}
\address{Amaury Lambert: Laboratoire de Probabilit\'es, Statistique \& Mod\'elisation,
Sorbonne Universit\'e, France.} \email{amaury.lambert@sorbonne.universite.fr}


\begin{document}
\maketitle

\begin{abstract}
We introduce a Poissonization method to study the coalescent structure of uniform samples from branching processes. This method relies on the simple observation that a uniform sample of size $k$ taken from a random set with positive Lebesgue measure may be represented as a mixture of Poisson samples with rate $\lambda$ and mixing measure $k \mathrm{d} \lambda/ \lambda$. We develop a multitype analogue of this mixture representation, and use it to characterise the coalescent structure of multitype continuous-state branching processes in terms of random multitype forests. Thereafter we study the small time asymptotics of these random forests, establishing a correspondence between multitype continuous-state branching proesses and multitype $\Lambda$-coalescents. 
\end{abstract}

\section{Introduction}

\subsection{Poissonization} \label{sec pois}
Suppose under a probability measure $P$ we have a branching process with continuous states. Consider the following problem:
\begin{quote}
Under an extension $\P^{k,T}$ of $P$,  sample uniformly and independently $k$ members of the population at time $T$. What can we say about the ancestors of the sampled  particles? What does their genealogical tree look like? What spatial position or type did these ancestors have?
\end{quote}
It turns out the problem of describing the ancestral tree of the $k$ uniformly chosen particles is difficult to attack directly, and it is useful to take an indirect approach. Consider instead the alternative problem:
\begin{quote}
Under a different extension $\Q^{\lambda,T}$ of $P$, we sample particles from the time $T$ population according to a Poisson process of rate $\lambda$. What can we say about the ancestors of these sampled particles?
\end{quote}
The behaviour of the ancestors of sampled particles is far easier to describe under $\Q^{\lambda,T}$ than it is under $\P^{k,T}$, due to the fact that under $\Q^{\lambda,T}$ the Poissonian structure of the sampling in some sense agrees with the branching structure of the process. 
Indeed, we will see later that under $\Q^{\lambda,T}$, the ancestors of sampled particles living at earlier times $t < T$ themselves occur within the time $t$ population according to a Poisson process.

The central idea of this paper is an integral formula allowing us to tackle the former fixed sample size problem involving $\mathbb{P}^{k,T}$ in terms of the latter Poissonian problem involving $\mathbb{Q}^{\lambda,T}$. To this end, let $\mathcal{S}$ be the number of particles sampled (so that under $\P^{k,T}$ we have $\P^{k,T} ( \mathcal{S} = k) = 1$, but under $\Q^{\lambda,T}$, the size $\mathcal{S}$ of the sample is random). It is clear that for events $A$, the map
\begin{align*}
A \mapsto \Q^{\lambda,T}( A, \mathcal{S} = k)
\end{align*}
is a measure supported on the event $\{\S = k\}$. Theorem \ref{po thm} is a somewhat informal statement of the one-dimensional special case of our main Poissonization result, Theorem \ref{thm:po c}, stating that $\P^{k,T}$ may be realised as a mixture of the measures $\mathbb{Q}( \cdot , \mathcal{S} = k )$.

\begin{thm} \label{po thm}
For any event $A$ which depends on the process and the ancestry of the sampled particles, we have
\begin{align} \label{po eq}
\P^{k,T} (A )= \int_0^\infty \frac{ k  \mathrm{d} \lambda}{ \lambda} \Q^{\lambda,T} \left( A, \mathcal{S} = k \right).
\end{align}
\end{thm}

The value of Theorem \ref{po thm} lies in the fact that in many applications $\Q^{\lambda,T}(A, \mathcal{S} = k)$ is far easier to calculate than $\P^{k,T}(A)$, and furthermore, the integrated expression is the most simplified expression possible. 

Before discussing our results further, in the next section we look at a quick concrete example: What is the probability that $k$ uniformly chosen particles from a continuous-state branching process are descended from the same time $0$ ancestor?

\subsection{A quick example} \label{sec csbp ex}
A one-dimensional continuous-state branching process (CSBP) is a $\mathbb{R}_{ \geq 0}$-valued Markov process $(Z(t))_{t \geq 0}$ enjoying the branching property: if $Z^x$ is a copy of the process starting from $x$ and $Z^y$ is an independent copy of the process starting from $y$, then $Z^x + Z^y$ has the same law as $Z^{x+y}$. The branching property implies the existence of a function $u(t,\lambda)$ such that if $P_x$ is the law of the process starting from $x  > 0$, then
\begin{align*}
P_x [ e^{ - \lambda Z(t) } ] = e^{ - x u(t, \lambda) }.
\end{align*}
CSBPs may be endowed with a notion of genealogy, so that we may associate each interval $(0,Z(t))$ with a set $\mathcal{Z}(t)$ of `particles alive at time $t$', in such a way that whenever $s < t$, each particle in $\mathcal{Z}(t)$ has a unique ancestor particle in $\mathcal{Z}(s)$. Consider now the following problem. Under a probability measure $\mathbb{P}_x^{k,T}$, sample $k$ particles uniformly and independently from $(0,Z(T))$. We would like to calculate $\mathbb{P}_x^{k,T}(A)$, where 
\begin{align*}
A := \{ \text{The $k$ sampled points of $(0,Z(T))$ are descended from the same ancestor in $(0,Z(0))$} \}.
\end{align*}
It turns out it is difficult to compute $\mathbb{P}_x^{k,T}(A)$ directly. Under $\mathbb{Q}_x^{\lambda,T}$ however, by utilising the branching property in conjunction with the independent increments of a Poisson process, it is possible to show that the probability that a Poisson process of rate $\lambda$ on $(0,Z(T))$ has $k$ points, and that all of these points are descended from the same ancestor in $(0,Z(0))$, is given by 
\begin{align} \label{eq:pok}
\mathbb{Q}_x^{\lambda,T} \left( \mathcal{S} =k , A \right) =  (-1)^{k-1} \lambda^k \frac{x e^{ - x u(T,\lambda)}}{k!}  \frac{ \partial^k}{ \partial \lambda^k }   u(T,\lambda).  
\end{align}
The formula \eqref{eq:pok} for the `Poissonized' version of the common ancestor problem gives us an immediate integral formula for $\mathbb{P}_x^{k,T}(A)$. Indeed, plugging \eqref{eq:pok} into Theorem \ref{po thm}, we obtain
\begin{align} \label{k mrca}
\P_x^{k,T} \left( A  \right) = \frac{ (-1)^{k-1}x }{ k! }  \int_0^\infty  \frac{  k  \mathrm{d} \lambda}{ \lambda} \lambda^k e^{ - x u(T,\lambda)} \frac{ \partial^k}{ \partial \lambda^k }  u(T,\lambda).  
\end{align}
Equation \eqref{k mrca} is new, extending the special case $k=2$ appearing in Corollary 1 of Lambert \cite{lam}. We believe the integrated form  in \eqref{k mrca} to be the most simplified equation for this probability, which we hope the reader will regard as a testament to the suitability of this indirect approach. In fact, \eqref{k mrca} is a special case of the far more general Theorem \ref{thm:CSBP2}.

\subsection{The coalescent structure of CSBPs: two approaches}

The focal point of the present paper is the application of Poissonization methods to the coalescent structure of multidimensional CSBPs. Before stating our results in full in Sections 2, 3 and 4, we take a moment to outline two different approaches to understanding these coalescent structures:
\begin{itemize}
\item The first approach is the \emph{distributional} approach, where we seek to obtain explicit formulas for the finite dimensional distributions of the coalescent processes associated with samples of particles chosen at a time $T$. This approach is overviewed in Section \ref{subsec:distributional} 
\item The second approach is the \emph{local} approach, where we seek to understand coalescent processes through their event  rates, in particular drawing on connections with the Pitman-Sagitov $\Lambda$-coalescents. Here the particles are sampled at a very small time $t$, and we look at the $t \to 0$ asymptotics for probabilities of the various events that may happen in the small time window $[0,t]$. This approach is overviewed in Section \ref{subsec:local}.
\end{itemize}

\subsection{The distributional approach to the  coalescent structure of CSBPs} \label{subsec:distributional} 
Suppose under a probability measure $P_x$ we have a $d$-dimensional continuous-state branching process $\left(Z(t) \right)_{t \geq 0}$ starting from $x \in \mathbb{R}_{ \geq 0}^d$, and run until time $T$. We think of the vector $Z(t)$ as capturing the size of a multitype population at time $t$, with the $i^{\text{th}}$ component $Z_i(t)$ representing the size of the $i^{\text{th}}$ type.

Under a probability measure $\mathbb{P}_x^{k,T}$ extending $P_x$, consider taking a sample $\textbf{b} := \left( b_{i,j} : 1 \leq i \leq d , 1 \leq j \leq k_i \right)$ of particles from the time $T$ population, where $k = (k_1,\ldots,k_d)$ is a vector of non-negative integers, and for each $i$, the row $(b_{i,1},\ldots,b_{i,k_i} )$ constitutes a sample of $k_i$ particles taken uniformly and independently from the type $i$ population at time $T$. In other words, for each pair $(i,j)$, $b_{i,j}$ is a real number uniformly distributed on the interval $(0,Z_i(T))$, where $Z_i(T)$ is the size of the type $i$ population at time $T$.

Let $\mathsf{t} := (t_0,t_1,\ldots,t_m)$ be a mesh of $[0,T]$ (that is  $ = t_0 < t_1 < \ldots < t_{m-1} < t_m = T$), and consider that for each $0 \leq \ell \leq m-1$, each particle $b_{i,j}$ had an ancestor of some type at time $t_\ell$. This ancestry creates an `ancestral forest' $\mathrm{For}_{\mathsf{t}}(\textbf{b})$ of $\textbf{b}$ across the mesh $\mathsf{t}$. We give a full description of the forest in Section 3, but let us just say here that the roots of the trees in the forest correspond to time $0$ ancestors of the sampled particles, and the vertices of the tree carry labels in $\{1,\ldots,d\}$ corresponding to the types of particles they represent.

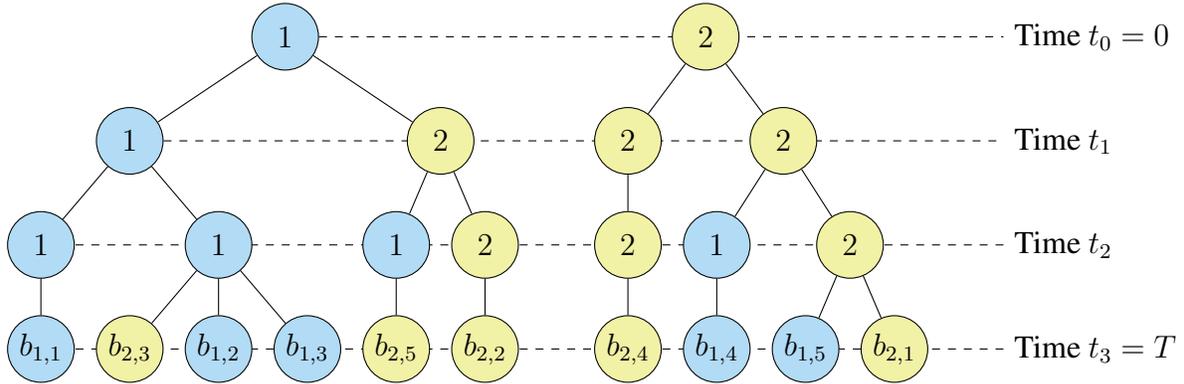
\begin{figure}[ht] 

\begin{forest}
[, phantom, s sep = 1cm
 [$1$,fill=col1,name=s,
[$1$,fill=col1,name=r,
[$1$,fill=col1, name=r2,
[$b_{1,1}$,fill=col1,name=r3,]
]
[$1$,fill=col1
[$b_{2,3}$,fill=col2]
[$b_{1,2}$,fill=col1]
[$b_{1,3}$,fill=col1]
]
]
[$2$,fill=col2,
[$1$,fill=col1
[$b_{2,5}$,fill=col2]
]
[$2$,fill=col2
[$b_{2,2}$,fill=col2]
]
]
]
 [$2$,fill=col2, name=g0,
[$2$,fill=col2
[$2$,fill=col2, 
[$b_{2,4}$,fill=col2]
]
]
[$2$,fill=col2, name=g1,
[$1$,fill=col1
[$b_{1,4}$,fill=col1]
]
[$2$,fill=col2, name=g2,
[$b_{1,5}$,fill=col1]
[$b_{2,1}$,fill=col2,name=g24,]
]
]
]
]
     \node (b0) [right=of g0 -| g24.east] {Time $t_0 = 0$};
    \node (b) [right=of g1 -| g24.east] {Time $t_1$};
    \begin{scope}[on background layer]
        \draw[dashed,] (r) -- (b);
    \end{scope}
\node(c) [above=of r,left=of s] {};
  \begin{scope}[on background layer]
        \draw[dashed,] (s) -- (b0);
    \end{scope}
\node(b2) [right=of g2 -| g24.east] {Time $t_2$};
  \begin{scope}[on background layer]
        \draw[dashed,] (b2) -- (r2);
    \end{scope}
\node(b3) [right=of g24 -| g24.east] {Time $t_3 = T$};
  \begin{scope}[on background layer]
        \draw[dashed,] (b3) -- (r3);
    \end{scope}
\end{forest}

  \caption{A $k = (5,5)$ sample of particles from a two-dimensional continuous-state branching process at a time $T$, and the ancestral forest of this sample across the mesh $(t_0,t_1,t_2,t_3)$. The particles $b_{2,3}, b_{1,2}$ and $b_{1,3}$ are all descendants of the same ancestor of type $1$ in the $t_2$ population.
 }
\label{fig:tree1}
\end{figure}

Our goal is to understand the law of the random forest $\mathrm{For}_{\mathsf{t}}(\textbf{b})$ under the probability measure $\mathbb{P}_x^{k,T}$. Without going into too much detail at this stage, we find it easier initially to understand the law of these random forests under a Poissonized version $\mathbb{Q}_x^{\lambda,T}$ of the measure $\mathbb{P}_x^{k,T}$, where the type $i$ particles are sampled according to a Poisson process of rate $\lambda_i$. We ultimately find in Theorem \ref{thm:CSBP1} that under such measures we have 
\begin{align} \label{eq:Qmention}
\Q_x^{\lambda,T} \left( \mathrm{For}_{\mathsf{t}}(\textbf{b})  = \mathcal{H} \Big| \mathcal{S} = k \right) = \frac{ \mathsf{E}_{x,\mathsf{t}} \left( \mathcal{H} , \lambda \right)  }{ \sum_{ \mathcal{I} \in \mathbb{H}^k(m) } \mathsf{E}_{x,\mathsf{t}} \left( \mathcal{I} , \lambda \right) },
\end{align}
where $\mathbb{H}^k(m)$ is the set of possible ancestral forests with leaves $\left(b_{i,j} : 1 \leq i \leq d, 1 \leq j \leq k_i \right)$ and $m$ generations, and $\mathsf{E}_{x,\mathsf{t}}( \cdot , \lambda)$ is a non-negative `energy' function on labelled forests depending on the initial population size $x$, the mesh $\mathsf{t}$, and the sampling rate $\lambda \in \mathbb{R}_{\geq 0}^d$. This energy function has an explicit form in terms of a product over nodes of the forest in terms involving the derivatives of the Laplace exponent $u(t,\lambda)$ (see \eqref{eq:energy}).

By combining Theorem \ref{thm:CSBP1} with a multidimensional version of the Poissonization result seen in the introduction, we obtain Theorem \ref{thm:CSBP2}, which states that there is a probability measure $ \Pi^k_T ( \mathrm{d} \lambda) $ on $\mathbb{R}_{ \geq 0}^d$ such that 
\begin{align*}
\mathbb{P}_x^{k,T} \left(\mathrm{For}_{\mathsf{t}}(\textbf{b}) = \mathcal{H}   \right) = \int_{\mathbb{R}_{ \geq 0}^d } \Pi^k_T ( \mathrm{d} \lambda)  \frac{ \mathsf{E}_{x,\mathsf{t}} \left( \mathcal{H} , \lambda \right)  }{ \sum_{ \mathcal{I} \in \mathbb{H}^k(m) } \mathsf{E}_{x,\mathsf{t}} \left( \mathcal{I} , \lambda \right) }.  
\end{align*}

One of the key tools in our approach here is a connection with a recent multi-variate version of Fa\`a di Bruno's formula appearing in Johnston and Prochno \cite{JP}, giving a formula for the higher order derivatives of composition chains
\begin{align*}
\frac{ \partial^{k_1} }{ \partial \lambda_1^{k_1} } \ldots \frac{ \partial^{k_d} }{ \partial \lambda_1^{k_d} } f \circ F^{(1)} \circ \ldots \circ F^{(m)} (\lambda)  = \sum_{ \mathcal{H} \in \mathbb{H}^k(m) } \mathsf{E}_{f,F^{(i)}}( \mathcal{H} , \lambda),
\end{align*} 
in terms of a partition function over labelled forests, where $f: \mathbb{R}^d \to \mathbb{R}$ and $F^{(i)} : \mathbb{R}^d \to \mathbb{R}^d$ are smooth functions. This multi-variate Fa\`a di Bruno formula may be used in conjunction with the semigroup identity for the Laplace exponent $u(t,\lambda)$ to obtain a compact expression for the denominator occuring in \eqref{eq:Qmention}.

\subsection{The local coalescent structure of CSBPs} \label{subsec:local}

We now discuss the local approach to coalescence in CSBPs. For the sake of clarity, and to motivate ideas, we begin by giving a full overview of the special one-dimensional case, $d=1$, where our local result is already implicit in the work of Labb\'e \cite{labbe}. Thereafter we sketch how this statement may be used to deduce prior work on specific CSBPs by Bertoin and Le Gall \cite{BLG}, Donnelly and Kurtz \cite{DK}, and  Birkner et al. \cite{birknerplus}. 

\subsubsection{The one-dimensional case}
Every one-dimensional CSBP is characterised by a branching mechanism $\psi$, a convex function of the form 
\begin{align} \label{cmech}
\psi(\lambda) = - \kappa \lambda + \beta \lambda^2 + \int_0^\infty (e^{ - \lambda r} - 1 + \lambda r \ind_{r \leq 1} ) \nu(\mathrm{d}r), 
\end{align}
where $\kappa \in \mathbb{R}$, $\beta \geq 0$, and $(1 \wedge r^2) \nu(\mathrm{d}r)$ is a finite measure on $[0,\infty)$. The branching mechanism $\psi$ determines the Laplace exponent $u(t,\lambda)$, and hence the law of the process, via the partial differential equation
\begin{align*}
\frac{ \partial}{ \partial t } u(t,\lambda) + \psi\left( u(t,\lambda) \right ) = 0, \qquad u(0,\lambda )= \lambda.
\end{align*}
Recall that $\mathcal{Z}(t)$ denotes the set of particles alive at time $t$, and for positive integers $k$ let $\mathcal{P}_k$ be the set of partitions of $\{1,\ldots,k\}$. Following Bertoin and Le Gall \cite{BLG}, we may define a $\mathcal{P}_k$-valued process $(\pi_u)_{0 \leq u \leq t}$ associated with the CSBP $(Z(u))_{0 \leq u \leq t}$ as follows. Conditional on $\{ Z(t) > 0 \}$, pick $k$ particles $b_1,\ldots,b_k$ uniformly and independently from $\mathcal{Z}(t)$. For times $0 \leq u \leq t$, we define the random partition $\pi_u$ of $\{1,\ldots,k\}$ corresponding to the equivalence relation
\begin{align*}
\text{ $i \sim_{\pi_u} j$} \iff \text{$b_i$ and $b_j$ are descended from the same particle in $\Zc(t-u)$.}
\end{align*}
The resulting process $(\pi_u)_{0 \leq u \leq t}$ is an exchangeable coalescent process, meaning that its law is invariant under permutations of $\{1,\ldots,k\}$ and that the blocks of the process merge as time passes.

An important class of exchangeable and Markovian coalescent processes are the $\Lambda$-coalescents \cite{pit,sag}, which are characterised in terms of finite measures $\Lambda$ on $[0,1]$ through the property that when the process has $k$ blocks, any $j$ of these blocks are merging to form a single new block at rate
\begin{align*}
\int_0^1  s^{j-2} (1-s)^{k-j} \Lambda(\mathrm{d}s) .
\end{align*}

The coalescent process $(\pi_u)_{0 \leq u \leq t}$  associated with the general continuous-state branching process is non-Markovian, and therefore cannot have a direct representation as a $\Lambda$-coalescent. However, the process $(\pi_u)_{0 \leq u \leq t}$ \emph{is} exchangeable, and we now give a result stating that $(\pi_u)_{0 \leq u \leq t}$ has a local representation in terms of a $\Lambda$-coalescent, where the merger measure $\Lambda^\psi(x, \mathrm{d}s)$ depends on the population size and branching mechanism. 

Before describing this correspondence we need the following definition relating measures on $[0,\infty)$ to those on $[0,1)$. For $x > 0$ let $T_x : \mathbb{R}_{ \geq 0} \to [0,1)$ be the map $T_x(r) = \frac{r}{x+r}$. Given a measure $\nu$ on $[0,\infty)$, we define the pushforward $T_x^\# \nu$ of $\nu$ to be the measure on $[0,1)$ given by 
\begin{align*}
T_x^\# \nu( A) =\nu \left( T_x^{ -1}(A) \right). 
\end{align*}
Theorem \ref{single type thm} is the special one-dimensional case of our main result on the local coalescent structure of multidimensional CSBPs, providing a dynamic link between continuous-state branching processes and $\Lambda$-coalescents. The merger measure at each time depends on the population size $x$ and the ingredients in the L\'evy-Khintchine representation \eqref{cmech} through the constant $\beta > 0$ and the pushforward of $\nu$ onto $[0,1]$ through $T_x$ but is independent of the deterministic growth parameter $\kappa$. 

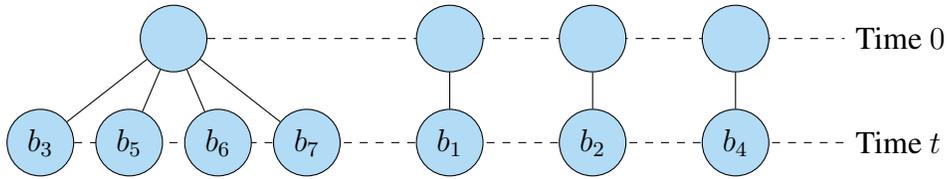
\begin{figure}[ht]

\begin{forest}
[, phantom, s sep = 1cm
[,fill=col1, name=a0,
[$b_3$,fill=col1, name=b0,]
[$b_5$,fill=col1,]
[$b_6$,fill=col1,]
[$b_7$,fill=col1,]
]
[,fill=col1,
[$b_1$,fill=col1,]
]
[,fill=col1,
[$b_2$,fill=col1,]
]
[,fill=col1,name=a,
[$b_4$,fill=col1,name=b,]
]
]
     \node (a1) [right=of a] {Time $0$};
    \node (b1) [right=of b] {Time $t$};
    \begin{scope}[on background layer]
        \draw[dashed,] (a0) -- (a1);
    \end{scope}
    \begin{scope}[on background layer]
        \draw[dashed,] (b0) -- (b1);
    \end{scope}
\end{forest}

  \caption{A sample of $k = 7$ particles taken from a one-dimensional continuous-state branching process at a small time $t$, of which $j=4$ are recently descended from the same ancestor at time $0$.
 }\label{fig:tree short}
\end{figure}

\begin{thm}[Labb\'e \cite{labbe}] \label{single type thm}
Suppose under a probability measure $P_x$ we have a continuous-state branching process starting from $x > 0$. Then for small times $t$, under a probability measure $\P_x^{k,t}$ extending $P_x$, pick $k$ particles uniformly from the population of the process at time $t$, and let $\left( \pi_u \right)_{0 \leq u \leq t}$ be the associated partition process. For $j \geq 2$, let $\gamma_{k,j}$ be a partition of $\{1,\ldots,k\}$ with one block of size $j$ and the remaining blocks all singletons. Then
\begin{align*}
\lim_{t \to 0} \frac{1}{t} \P^{k,t}_x \left( \pi_t = \gamma_{k,j} \right) = \int_0^1 s^{ j - 2} (1-s)^{ k- j} \Lambda^\psi( x,\mathrm{d}s),
\end{align*}
where the merger measure $\Lambda^\psi(x, \mathrm{d}s)$ is given by
\begin{align} \label{csbp lambda}
\Lambda^\psi( x ,  \mathrm{d}s)  = \frac{ 2 \beta}{ x} \delta_0 (\mathrm{d}s) + x s^2 T_x^\# \nu(\mathrm{d}s),
\end{align}
where $\beta$ is the Brownian component of the CSBP and $\nu$ is its L\'evy measure, so that its L\'evy-Khintchine representation is as in \eqref{cmech}.
\end{thm}

Theorem \ref{single type thm} connects a single type CSBP with a $\Lambda$-coalescent in such a way that the merger measure $\Lambda$ depends on $\beta$ and $\nu$ but is independent of the deterministic growth parameter $\kappa$. This independence is easily interpreted, since the deterministic growth has no effect on the local coalescent structure of the process (though it does effect the size of the process in the long run). As mentioned above, Theorem \ref{single type thm} is implicit in the work of Labb\'e \cite[Section 3]{labbe}, though the precise formulation given here in terms of the $t \to  0$ limit may be deduced by setting $d=1$ in Theorem \ref{thm:rate}, which we state and prove in the sequel. 

Several preceeding connections between specific CSBPs and $\Lambda$-coalescents have already 
appeared in the literature. Arguably the simplest CSBP is Feller's diffusion, the CSBP with branching mechanism given by $\psi(\lambda) = \frac{1}{2} \lambda^2$, and corresponding to the case where $\beta = \frac{1}{2}$ and $\nu = 0$ in the L\'evy-Khintchine representation \eqref{cmech}. Here, the formula \eqref{csbp lambda} tells us that the associated merger measure is given by 
\begin{align} \label{kinlambda}
\Lambda^\psi(x,\mathrm{d}s) = \frac{ 1}{ x} \delta_0(\mathrm{d}s) .
\end{align}
The $\Lambda$-coalescent corresponding to the case where $\Lambda = \delta_0$ is known as Kingman's coalescent, and is characterised by only having pairwise mergers. Indeed, \eqref{kinlambda} ties in with a result by Donnelly and Kurtz \cite[Theorem 5.1]{DK}, which loosely speaking states that the coalescent process $(\pi_t)_{0 \leq t \leq T}$ associated with the Feller diffusion is a Kingman coalescent run at a speed reciprocal to $Z_{T-t}$ at time $t$. In Corollary \ref{cor:DK} we will encounter a multitype analogue of this result.

When $\psi(\lambda) = \lambda \log \lambda$, we call the CSBP a \emph{Neveu process}. This corresponds to the case where $\beta=0$ and $\nu(\mathrm{d}r) = r^{-2}$ in \eqref{cmech}. Plugging $\beta = 0$ and $\nu(\mathrm{d}r) = r^{-2}$ into \eqref{csbp lambda} we obtain
\begin{align} \label{neveubs}
\Lambda^\psi(x,\mathrm{d}s) = \mathrm{d}s. 
\end{align} 
We emphasise that this quantity is independent of the population size $x$, and that \eqref{neveubs} leads us to a result by Bertoin and Le Gall \cite{BLG}, who showed that the coalescent process $(\pi_t)_{0 \leq t \leq T}$ associated with Neveu's CSBP is given by the \emph{Bolthausen-Sznitman coalescent} \cite{BS}, the $\Lambda$-coalescent corresponding to the case where $\Lambda$ is the Lebesgue measure on $[0,1]$.

A priori, there is a surprising qualitative difference in these two previous results regarding the coalescent structure of the Feller diffusion and of the Neveu process: the coalescent process associated with the former has strong dependence on the population size, whereas that of the latter process is independent of the population size. Birkner et al \cite{birknerplus} gave an explanation for this qualitative discrepancy by bridging the gap between these two results, establishing a connection between the so-called $\alpha$-stable CSBP with L\'evy measure $\nu(\mathrm{d}r) = C_\alpha \mathrm{d}r/r^{ 1 + \alpha}$ and the Beta coalescents, whose merger measure is given by $\Lambda_\alpha(\mathrm{d}s) := \tilde{C}_\alpha x^{1 - \alpha} \left( \frac{s}{1-s} \right)^{ 1 - \alpha} \mathrm{d}s$. Loosely speaking, this connection interpolates between the aforementioned results of Donnelly and Kurtz \cite{DK} and Bertoin and Le Gall \cite{BLG}, with the case $\alpha = 2$ corresponding to the Feller diffusion and Kingman's coalescent and $\alpha =1$ corresponding to Neveu's CSBP and the Bolthausen-Sznitman coalescent. 

Let us also refer the reader to Berestycki, Berestycki and Schweinsberg \cite{BBS}, who gain a great deal of information about the small time information of Beta coalescents by embedding them into alpha-stable CSBPs, as well as Berestycki, Berestycki and Limic \cite{BBL}, who use a variation of the Donnelly and Kurtz lookdown construction (\cite{DK}) to establish small-time connections between general CSBPs and $\Lambda$-coalescents.

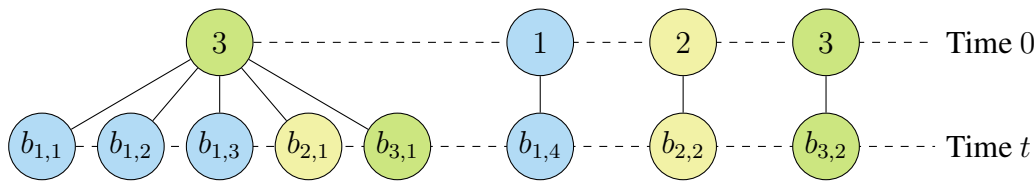
\begin{figure}[ht!]

\begin{forest}
[, phantom, s sep = 1cm
[$3$,fill=col3, name=a0,
[$b_{1,1}$,fill=col1, name=b0,]
[$b_{1,2}$,fill=col1,]
[$b_{1,3}$,fill=col1,]
[$b_{2,1}$,fill=col2,]
[$b_{3,1}$,fill=col3,]
]
[$1$,fill=col1,
[$b_{1,4}$,fill=col1,]
]
[$2$,fill=col2,
[$b_{2,2}$,fill=col2,]
]
[$3$,fill=col3,name=a,
[$b_{3,2}$,fill=col3,name=b,]
]
]
     \node (a1) [right=of a] {Time $0$};
    \node (b1) [right=of b] {Time $t$};
    \begin{scope}[on background layer]
        \draw[dashed,] (a0) -- (a1);
    \end{scope}
    \begin{scope}[on background layer]
        \draw[dashed,] (b0) -- (b1);
    \end{scope}
\end{forest}

  \caption{A sample of $k = (4,2,2)$ particles taken from a three-type CSBP at a small time $t$. At time zero, a collection of $j = (3,1,1)$ particles were recently descended from a type $c=3$ ancestor.
 }\label{fig:tree2}
\end{figure}

As stated above, we in fact work in the general multidimensional setting, giving  an expression for the merger rates associated with multidimensional CSBPs in terms of multidimensional $\Lambda$-coalescents. Theorem \ref{thm:rate}, which is the multidimensional version of Theorem \ref{single type thm} above, states that when a sample of $(k_1,\ldots,k_d)$ individuals is taken just after a moment when the population has size $x \in \mathbb{R}_{\geq 0}^d$, the probability that a group of blocks with types $(\alpha_1,\ldots,\alpha_d)$ merges to form a single block of type $c$ is given by 
\begin{align*}
\sum_{ j = 1, j \neq c}^d \ind_{ \alpha = \mathbf{e}_j} \kappa_{c,j} \frac{x_c}{x_j} +  \ind_{ \alpha = 2 \mathbf{e}_c} \frac{ 2 \beta_c }{x_c} + x_c \int_{ [0,1]^d} s^\alpha ( 1- s)^{ k- \alpha} T_x^{\#} \nu_c( \mathrm{d} s).
\end{align*}
where $s^{ \alpha } (1 - s)^{k - \alpha } := \prod_{ i = 1}^d s_i^{\alpha_i}  (1 - s_i)^{k_i - \alpha_i}$, $T_x^{\# } \nu_c$ is a pushforward of the $c^{\text{th}}$ component of the L\'evy measure $\nu_c$ on $[0,\infty)^d$ to $[0,1)^d$, and $\beta_c$ and $\kappa_{c,j}$ are parameters associated with the Brownian and deterministic parts of the process respectively.

Finally, let us discuss the application of Theorem \ref{thm:rate} to the multidimensional Feller diffusion, the $d$-dimensional continuous-state branching process with branching mechanism 
\begin{align} \label{feller mech}
\psi_i( \lambda ) = - \sum_{ j = 1}^d \kappa_{i,j} \lambda_j + \beta_i \lambda_i^2 \qquad 1 \leq i \leq d,
\end{align}
where $\beta_i \geq 0$, and the $\kappa_{i,j}$ are required to be non-negative whenever $i \neq j$. This continuous-state branching process has a version with continuous paths, and may be understood as a strong solution to the stochastic differential equation
\begin{align} \label{feller SDE}
\mathrm{d} Z_j(t) =  \sqrt{ 2  \beta_j Z_j(t) } \mathrm{d}B_j(t)  + \sum_{ i = 1}^d \kappa_{i,j} Z_i(t)  \mathrm{d}t, 
\end{align}
where $\left( B_i(t) : 1 \leq i \leq d \right)$ are independent standard Brownian motions. From \eqref{feller SDE} it is straightforward to interpret the branching structure of the process: the type $j$ population varies according to its own internal branching structure with variance $\beta_j$, and mass from the population of type $i$ continuously creates new mass of type $j$ at rate $\kappa_{i,j}$ per unit of type $i$ mass.

We now give a brief verbal account of Corollary \ref{cor:DK}, which amounts to an application of Theorem \ref{thm:rate} to the $d$-dimensional Feller diffusion. Consider sampling various particles from the different populations of a multidimensional Feller diffusion at some time $T$. Then going backwards in time, ancestral lines coalesce according to the following rates when the population size is given by $x \in \mathbb{R}_{ \geq 0}^d$:
\begin{itemize}
\item two type $i$ ancestral lines coalesce to form a single ancestral line of type $i$ at rate $2 \beta_i / x_i$,
\item a single ancestral line of type $i$ changes type to type $c$ at a rate $\kappa_{c,i} x_c/x_i$. 
\end{itemize}

In the next section we overview related literature.

\subsection{Further discussion of related work}
Many authors have studied the problem of describing the ancestry of $k$ uniformly chosen particles from the one-dimensional Bienaym\'e-Galton-Watson tree. 
The case $k=2$ appears in \cite{buhler:super,lam,oconnell:genealogy_mrca,zubkov:mrca}, with $k=3$ also considered in \cite {durrett:genealogy}. The general $k$ case has only been looked at more recently, by Grosjean and Huillet \cite{grosjean_huillet:coalescence}, Hong \cite{hong}, Le \cite{le:coalescence_GW}, and the first author and coauthors \cite{HJR,J17}. Let us point out in particular that Theorem 3.1 of \cite{J17} appears to play the role of a one-dimensional analogue to Theorem \ref{thm:CSBP1} of the present paper --- though in the setting of Galton-Watson trees as opposed to continuous-state branching processes. Popovic \cite{popovic:asymptotic_genealogy_critical_bp} looked at critical Bienaym\'e-Galton-Watson trees, taking a different viewpoint.
Namely, she considered the genealogy of a sample containing a proportion $p \in (0,1]$ of the entire population in a tree conditioned to survive until a large time, relating the coalescence times to a point process. Popovic's sampling here corresponds to the one-dimensional measure $\mathbb{Q}^p$ we define in the discrete Poissonization result, Theorem \ref{thm:po d}. Let us also mention other work in this vein by Popovic and Rivas \cite{pr}, as well as by the second author and his coauthors \cite{lambert:coalescent_branching_trees,lambert:genealogy_binary_branching_process,lambert_stadler:birth_death_cpp}. 

As for the branching processes with continuous-states, the rigorous construction of the genealogy dates back to Le Gall and Le Jan \cite{LGLJ}, and ever since various related problems have been tackled. Labb\'e \cite{labbe} studied the so-called `Eve property' of CSBPs - the existence of an ancestor from which the entire population descends asymptotically, the second author looked at CSBPs with immigration in \cite{lambert:genealogy_csbp_imm}, and Bertoin, Fontbona and Martinez \cite{bertoin_fontbona_martinez} considered `prolific individuals' associated with the process. The reader is also referred to Section 4 of the set of lecture notes \cite{lambert dynamics} by the second author.

Finally, we would like to highlight the recent work of Foucart, Ma, and Mallein \cite{FMM}, which is of particular relevance. The authors of \cite{FMM} studied certain flows associated with CSBPs going backwards in time, in particular those associated with Poisson samples on the populations of one-dimensional CSBPs. Namely, consider running a rate $\lambda$ Poisson process on the population of a one-dimensional CSBP at time $T$, and declare a particle in the time zero population to have \emph{outdegree $j$} if they have exactly $j$ descendants at time $T$ who are points of the Poisson process. Paraphrasing slightly, Foucart, Ma and Mallein \cite[Section 5]{FMM} observed that particles with outdegree $j$ also occur within the time zero population according to a Poisson process, this time with rate
\begin{align*}
r^j(T,\lambda) := (-1)^{j-1} \frac{ \lambda^j}{ j!} \frac{ \partial^j }{ \partial \lambda^j } u(T,\lambda),
\end{align*}
where $u(T,\lambda)$ is the Laplace exponent of the CSBP. In Section 6 of the present article, we make a multidimensional analogue of this observation which is crucial in our study of the coalescent structure of multidimensional CSBPs. That concludes the discussion of related literature.

We now provide an overview of the remainder of the paper.

\subsection{Overview}

\begin{itemize}
\item In Section \ref{sec:po}, we state our Poissonization results in full detail, which provide integral transforms allowing one to `Poissonize' uniform samples from both continuous and discrete multidimensional populations.

\item
In Section \ref{sec:distributional} we state our results on the distributional approach to the coalescent structure of continuous-state branching processes. Theorem \ref{thm:CSBP1} shows that the random ancestral forests occuring under $\mathbb{Q}_x^{\lambda,T}$ may be understood in terms of partition functions.  By combining Theorem \ref{thm:CSBP1} with the multidimensional Poissonization result \ref{thm:po c}, we immediately obtain Theorem \ref{thm:CSBP2}, which gives an explicit integral formula for the law of a random ancestral forest chosen under $\mathbb{P}_x^{k,T}$.

\item
In Section \ref{sec:local} we state our results on the local approach to the coalescent structure of continuous-state branching processes in terms of multidimensional $\Lambda$-coalescents. 

\item
The remaining sections, Sections \ref{sec:po proof}, \ref{sec:distributional proof} and Sections \ref{sec:local proof}, are dedicated to proving the results stated in Sections \ref{sec:po}, \ref{sec:distributional} and \ref{sec:local} respectively.

\end{itemize}

\section{The Poissonization theorems} \label{sec:po}

In this section we state our Poissonization theorems, which give integral formulas allowing one to pass between difficult problems involving uniform sampling to easier problems involving 
Poissonized sampling. The first result, Theorem \ref{thm:po c}, is a multidimensional version of the theorem stated in the introduction. The latter result, Theorem \ref{thm:po d} is a discrete version of Theorem \ref{thm:po c}, suitable for sampling from discrete populations like multidimensional Galton-Watson trees or branching Brownian motion. In the discrete version, the `Bernoulli sample' plays the role of the `Poisson sample', where a Bernoulli sample refers to independently including each particle in the sample with probability $p$.

\subsection{Continuous Poissonization} \label{sec:cpoi}
A (continuous) $d$-type population is any subset of $\{1,\ldots,d\} \times (0,\infty)$ of the form
\begin{align*}
\mathcal{Z}  = \bigcup_{ i =1}^d \left( \{ i \} \times \mathcal{Z}_i \right),
\end{align*}
where each $\mathcal{Z}_i$ is a (possibly empty) measurable subset of $(0,\infty)$ with finite Lebesgue measure. We define the Lebesgue measure $\mathrm{Leb}(\mathcal{Z})$ to be the vector
\begin{align*}
\mathrm{Leb}(\mathcal{Z} ) = \left( \mathrm{Leb}(\mathcal{Z}_1),\ldots,\mathrm{Leb}(\mathcal{Z}_d) \right) \in \mathbb{R}_{ \geq 0}^d,
\end{align*}
where of course $\mathrm{Leb}(\mathcal{Z}_i)$ denotes the Lebesgue measure of $\mathcal{Z}_i$. By canonically associating $ \{ i \} \times \mathcal{Z}_i $ with $\mathcal{Z}_i$, there are natural notions of Lebesgue measure, uniform distribution, and Poisson process on $\{ i \} \times \mathcal{Z}_i$. 

We now define two different types of sampling on a continuous $d$-type population $\mathcal{Z}$. 
\begin{defn}[$k$-sample on $\mathcal{Z}$] \label{def:k}
Let $k$ be an element of $\mathbb{Z}_{ \geq 0}^d$. A $k$-sample on a $d$-type population $\mathcal{Z}$ is a random collection $\textbf{b} = (b_{i,j} : 1 \leq i \leq d, 1 \leq j \leq \mathcal{S}_i )$ of elements of $\mathcal{Z}$, such that independently for each $i$, if the Lebesgue measure of $\mathcal{Z}_i$ is non-zero we set $\mathcal{S}_i = k_i$ and sample $k_i$ points $b_{i,1},\ldots,b_{i,k_i}$ uniformly from $\{ i \} \times \mathcal{Z}_i$, whereas if the Lebesgue measure of $\mathcal{Z}_i$ is zero we set $\mathcal{S}_i = 0$ so that the $i^{\text{th}}$ row of $\textbf{b}$ is the empty array. 
\end{defn}

\begin{defn}[$\lambda$-sample on $\mathcal{Z}$] \label{def:lambda}
Let $\lambda$ be an element of $\mathbb{R}_{ \geq 0}^d$. A $\lambda$-sample on a $d$-type population $\mathcal{Z}$ is a random collection $\textbf{b} = (b_{i,j} : 1 \leq i \leq d, 1 \leq j \leq \mathcal{S}_i )$ of elements of $\mathcal{Z}$ such that independently for each $i$, $(b_{i,1},\ldots,b_{i,\S_i})$ are the points of a rate $\lambda_i$ Poisson process on $\{i\} \times \mathcal{Z}_i$ labelled according to a uniformly chosen permutation of $\{1,\ldots,\S_i\}$. 
\end{defn}

(While technically, there is some ambiguity in the two preceding definitions --- in that an element $\lambda$ of $\mathbb{R}_{\geq 0}^d$ may also be an element of $\mathbb{Z}_{\geq 0}^d$ --- it will always be clear from context what we mean when we say $\lambda$-sample.)

Suppose now we have a measurable space $(\Omega,\mathcal{F})$ carrying a random $d$-type population $\mathcal{Z}$.
Let $(\Omega', \mathcal{F}')$ be another measurable space carrying a random vector $\S = (\S_1,\ldots,\S_d)$ of non-negative integers and a collection of $\{1,\ldots,d\} \times (0,\infty)$-valued random variables $(b_{i,j} : 1 \leq i \leq d, 1 \leq j \leq \S_i )$. 

Let $P$ be a probability measure on $(\Omega,\mathcal{F})$, and consider the following two extensions of $P$ onto the product space $(\tO, \tF) := (\Omega \otimes  \Omega', \mathcal{F} \otimes \mathcal{F}')$.
\begin{itemize}
\item For $k \in \mathbb{Z}_{ \geq 0}^d$, define the extension $\mathbb{P}^k$ of $P$ by letting $\textbf{b}$ be a $k$-sample on $\mathcal{Z}$. 
\item For $\lambda \in \mathbb{R}_{ \geq 0}^d$, define the extension $\mathbb{Q}^\lambda$ of $P$ by letting $\textbf{b}$ be a $\lambda$-sample on $\mathcal{Z}$.
\end{itemize}
The measures $\P^k$ and $\Q^\lambda$ are extensions of $P$ in the sense that for all $A \in \mathcal{F}$ we have $\P^k (A \times \Omega') = \Q^\lambda( A \times \Omega' ) = P(A)$.

Finally, for non-negative integers $j$, define the measure $\pi^j( \mathrm{d} \lambda)$ on $[0,\infty)$ by
\begin{align} 
\pi^j( \mathrm{d} \lambda) :=
\begin{cases}
j  \mathrm{d} \lambda/ \lambda, \qquad  \text{if $j > 0$},\\
\delta_0( \mathrm{d} \lambda), \qquad \text{if $j = 0$}. 
\end{cases} \label{solo def}
\end{align}
The following theorem is our main Poissonization result and a multidimensional version of Theorem \ref{po thm} seen in the introduction, stating that $\P^k$ may be realised as a mixture of the measures $\Q^\lambda$. 
 
\begin{thm} \label{thm:po c}
For non-zero $k \in \mathbb{Z}^d_{ \geq 0}$, let $\pi^k( \mathrm{d} \lambda)$ be the product measure $\pi^k := \pi^{k_1} \otimes \ldots \otimes \pi^{k_d}$. Let $Z := \mathrm{Leb}(\mathcal{Z})$, and let $\{ Z \succ k \} $ denote the event that sampling is possible from the population, that is
\begin{align} \label{J def}
\{ Z \succ k \}  := \{ \text{The Lebesgue measure of $\mathcal{Z}_i$ is positive for each $i$ such that $k_i > 0$} \}.
\end{align}
Then for every $A$ in $\tilde{\mathcal{F}}$ 
\begin{align} \label{c multi eq}
\P^k (A, Z \succ k  ) = \int_{\mathbb{R}_{\geq 0}^d} \pi^k( \mathrm{d} \lambda) \Q^{\lambda}(A, \mathcal{S} = k ).
\end{align}
Moreover, $\pi^k( \mathrm{d} \lambda)$ is the unique measure with this property.
\end{thm}

The following result is an alternative statement of Theorem \ref{thm:po c}.

\begin{thm} \label{thm:c}
In the set up of Theorem \ref{thm:po c}, provided $P( Z \succ k ) > 0$ we may write
\begin{align*}
\mathbb{P}^k ( A | Z \succ k ) = \int_{\mathbb{R}^d_{ \geq 0}} \Pi^k(  \mathrm{d} \lambda) \Q^\lambda \left( A | \mathcal{S} = k \right),
\end{align*}
where 
\begin{align} \label{eq:laprob}
\Pi^k( \mathrm{d} \lambda) := P\left[ \frac{ (\lambda Z)^k }{ k!} e^{ - \langle \lambda, Z \rangle} \Big| Z \succ k \right]  \pi^k(  \mathrm{d} \lambda)  
\end{align}
is a probability measure on $\mathbb{R}_{ \geq 0}^d$, $Z := \mathrm{Leb}(\mathcal{Z})$, $\langle \lambda , Z \rangle := \sum_{i=1}^d \lambda_i Z_i$, and $\frac{(\lambda Z)^k}{k!} := \prod_{ i =1}^k (\lambda_i Z_i)^{k_i}/k_i!$.
\end{thm}

Theorem \ref{thm:c} has the interpretation that in order to take a uniform $k$-sample from a random $d$-type population, we may first choose a `random' $\lambda$ from $\mathbb{R}_{ \geq 0}^d$ according to the probability measure \eqref{eq:laprob}, and then condition on a $\lambda$-sample having $\{ \S = k \}$.

Theorem \ref{thm:po c} and Theorem \ref{thm:c} are proved in Section \ref{sec:cproofs}. In the next section we turn to discussing discrete Poissonization.
\subsection{Discrete Poissonization} \label{disc poiss}
In this section we introduce discrete Poissonization, a discrete analogue of the continuous Poissonization procedure suitable for sampling from discrete populations.

A discrete $d$-type population $\mathcal{N}$ is a finite set of the form
\begin{align} \label{eq:dunion}
\mathcal{N} = \bigcup_{ i = 1}^d \left(  \{i \} \times \mathcal{N}_i  \right)
\end{align}
where each $\mathcal{N}_i$ is a subset of $\{1,2,\ldots,\}$. We write $N_i$ for the number of elements of $\mathcal{N}_i$. 

We now consider two different ways of sampling from a discrete $d$-type population, first involving uniform sampling and then Bernoulli sampling.

\begin{defn}[$k$-sample on $\mathcal{N}$]
Let $k$ be an element of $\mathbb{Z}_{ \geq 0}^d$. A $k$-sample on a discrete $d$-type population $\mathcal{N}$ is a random collection $\textbf{b} = (b_{i,j} : 1 \leq i \leq d, 1 \leq j \leq \mathcal{S}_i )$ of elements of $\mathcal{N}$, such that independently for each $i$, if $N_i < k_i$ we set $\S_i = 0$, and if $N_i \geq k_i$, we set $\mathcal{S}_i = k_i$ and sample $k_i$ distinct points $b_{i,1},\ldots,b_{i,k_i}$ uniformly (i.e. uniformly without replacement) from $\{ i \} \times \mathcal{N}_i$.
\end{defn}

\begin{defn}[$p$-sample on $\mathcal{N}$]
Let $p$ be an element of $[0,1]^d$. A $p$-sample on a discrete $d$-type population $\mathcal{N}$ is a random collection $\textbf{b} := (b_{i,j} : 1 \leq i \leq d, 1 \leq j \leq \mathcal{S}_i )$ of elements of $\mathcal{N}$ generated as follows. Independently for each $(i,j)$, each element $(i,j)$ of $\mathcal{N}$ is included in the sample with probability $p_i$, and not included with probability $1-p_i$. Let $\mathcal{S}_i$ be the number of elements of the form $(i,j)$ included in the sample for some $j$. Then the elements $(i,j_1),\ldots,(i,j_{\S_i} )$ are labelled $b_{i,1},\ldots,b_{i,\S_i}$ according to a uniformly chosen permutation of $\{1,\ldots, \S_i\}$. 
\end{defn}

Now suppose we have a measurable space $(\Omega,\mathcal{F})$ carrying a random discrete $d$-type population. Let $(\Omega', \mathcal{F}')$ be another measurable space carrying a random vector $(\S_1,\ldots,\S_d)$ of non-negative integers and a collection of $\{1,\ldots,d\} \times \{1,2,\ldots\}$-valued random variables $(b_{i,j} : 1 \leq i \leq d, 1 \leq j \leq \S_i )$. 

Let $P$ be a probability measure on $(\Omega,\mathcal{F})$, and consider the following two extensions of $P$ onto the product space $(\tO, \tF) := (\Omega \otimes  \Omega', \mathcal{F} \otimes \mathcal{F}')$.
\begin{itemize}
\item For $k \in \mathbb{Z}_{ \geq 0}^d$, define the extension $\mathbb{P}^k$ of $P$ by letting $\textbf{b}$ be a $k$-sample on $\mathcal{N}$. 
\item For $p \in [0,1]^d$, define the extension $\mathbb{Q}^p$ of $P$ by letting $\textbf{b}$ be a $p$-sample on $\mathcal{N}$.
\end{itemize}
The measures $\P^k$ and $\Q^p$ are extensions of $P$ in the sense that for all $A \in \mathcal{F}$ we have $\P^k (A \times \Omega') = \Q^p( A \times \Omega' ) = P(A)$.

Finally, for non-negative integers $j$, define the measure $\bar{\pi}^j(dp)$ on $[0,1]$ by 
\begin{align*}
\bar{\pi}^j(d p) :=
\begin{cases}
j dp/ p, \qquad  \text{if $j > 0$},\\
\delta_0(d p), \qquad \text{if $j = 0$}.
\end{cases}
\end{align*}
The following theorem is a discrete analogue of Theorem \ref{thm:po c}, stating that $\P^k$ may be realised as a mixture of the measures $\Q^p$.

\begin{thm} \label{thm:po d}
For non-zero $k \in \mathbb{Z}^d_{ \geq 0}$, let $\bar{\pi}^k(dp)$ be the product measure $\bar{\pi}^k := \bar{\pi}^{k_1} \otimes \ldots \otimes \bar{\pi}^{k_d}$ on $\mathbb{R}_{\geq 0}^d$.  Let $\{ N \geq k \}  := \{ N_i \geq k_i \text{ for all $i$} \}$ denote the event that sampling is possible from the population. 
Then for every $A$ in $\tilde{\mathcal{F}}$
\begin{align} \label{d po eq}
\P^k (A, N \geq k) = \int_{[0,1)^d} \bar{\pi}^k(d p ) \Q^{p}(A, \mathcal{S} = k ).
\end{align}
Moreover, $\bar{\pi} ^k(d p)$ is the unique measure with this property.
\end{thm}

We remark that is possible to derive a discrete analogue of Theorem \ref{thm:c}. We leave the details to the reader.

In the next section, we turn to discussing the distributional approach to the coalescent structure of CSBPs.

\section{The distributional approach to the coalescent structure of CSBPs} \label{sec:distributional}

\subsection{Multi-type CSBPs} \label{sec:csbpintro}
 For multi-indices $\alpha$ and elements $\lambda$ of $\mathbb{R}_{\geq 0}^d$, we write $|\alpha| := \alpha_1 + \ldots + \alpha_d$, $\alpha! := \alpha_1! \ldots \alpha_d!$ and $\lambda^\alpha := \lambda_1^{\alpha_1} \ldots \lambda_d^{\alpha_d}$. We define the differential operator $D^\alpha: \mathcal{C}^\infty( \mathbb{R}^d) \to \mathcal{C}^\infty( \mathbb{R}^d)  $ by
\begin{align*}
D^\alpha f(\lambda) := \frac{ \partial^{\alpha_1}}{\partial \lambda_1^{\alpha_1} } \ldots  \frac{ \partial^{\alpha_d}}{\partial \la_d^{\alpha_d} } f(\la).
\end{align*}
Moreover, whenever $F:\mathbb{R}^d \to \mathbb{R}^d$ with components $(F_1,\ldots,F_d)$ where $F_i : \mathbb{R}^d \to \mathbb{R}$, we find it will be clearer notationally to write
\begin{align*}
D_i^\alpha [F] (\lambda) := D^\alpha F_i (\lambda).
\end{align*}
\subsubsection{Laplace exponents}
We now introduce finite dimensional continuous-state branching processes, following the set up in Kyprianou, Palau and Ren \cite{KPR}. A $d$-dimensional continuous-state branching process is a $\mathbb{R}^d_{\geq 0}$-valued Markov process enjoying the branching property: if $Z^x$ is a copy of the process starting from $x$ and $Z^y$ is a copy of the process starting from $y$, then $Z^x + Z^y$ has the same law as $Z^{x+y}$. The branching property implies the existence of a Laplace exponent $u: \mathbb{R}_{\geq 0} \times \mathbb{R}_{ \geq 0}^d \to \mathbb{R}_{\geq 0}^d$ such that if $P_x$ is the law of the process starting from $x \in \mathbb{R}_{ \geq 0}^d$ then
\begin{align} \label{branch exp}
P_x \left[ e^{ - \langle \theta , Z(t) \rangle } \right] = e^{ - \langle x, u(t,\theta) \rangle}
\end{align}
where $\langle x,y \rangle = \sum_{ i =1}^d x_i y_i$ denotes the Euclidean inner product on $\mathbb{R}^d$. It is straightforward to show using \eqref{branch exp} and the Markov property that the Laplace exponent satisfies the semigroup property
\begin{align} \label{semigroup}
u\left(t,u(s,\theta) \right) = u(t + s, \theta).
\end{align}
The Laplace exponent (and by extension the law of the continuous-state branching process) is determined by a branching mechanism $\psi: \mathbb{R}^d\to \mathbb{R}^d$ through the differential equation
\begin{align} \label{csbp pde multi}
\partial_t u(t,\theta) + \psi( u(t,\theta) ) = 0.                                 
\end{align}
As in the one-dimensional case, the branching mechanism $\psi:\mathbb{R}^d \to \mathbb{R}^d$ has a L\'evy-Khintchine representation, so that the $c^{\text{th}}$ component of $\psi$ is given by
\begin{align} \label{eq:lkrep}
\psi_c(\la) = - \sum_{ j = 1}^d \kappa_{c,j} \lambda_j + \beta_c \lambda_c^2 + \int_{\mathbb{R}_{\geq 0}^d} \left( e^{ - \langle \la , r \rangle } - 1 +  \lambda_c r_c \ind_{ r_c \leq 1 } \right) \nu_c(\mathrm{d}r),
\end{align}                                                                                                                                                                                                                                                                                                                                                                                                                                                                                                                                                                                                                                                                                                                                                                                                                                                                                                                                                                                                                                                                                                                                                                                                                                                                                                                                                                                                                                                                                                                                                                                                                                                                                                                                                                                                                                                                                                                                                                                                                                                                                                                                                                                                                                                                                                                                                                                                                                                                                                                                                                                                                                                                                                                                                                                                                                                                                                                                                                                                                                                                                                                                                                                                                                                                                                                                                                                                                                                                                                                                                                                                                                                                                                                                                                                                                                                                                                                                      
where $\kappa_{c,j}$ are real numbers with $\kappa_{c,j} \geq 0$ for $j \neq c$, $\beta_c \geq 0$ and the $\nu_c$ is a measure on $\mathbb{R}_{ \geq 0}^d$ satisfying the integrability condition 
\begin{align} \label{eq:integrability}
\int_{\mathbb{R}_{\geq 0}^d}\left\{ (r_c^2 \wedge 1 ) + \sum_{ j \neq c} (r_j \wedge 1) \right\}  \nu_c(\mathrm{d}r) < \infty.
\end{align} 
In summary, the distribution of the $d$-dimensional continuous-state branching process is determined
by a collection of real numbers $(\kappa_{c,j})_{ 1 \leq c,j \leq d}$ controlling the deterministic growth of the process, a collection of non-negative reals $(\beta_c)_{1 \leq c \leq d}$ controlling the Brownian fluctuations in each component, and a collection $(\nu_c)_{1 \leq c \leq d}$ of jump measures on $\mathbb{R}_{ \geq 0}^d$ determining the rates and sizes of discontinuities of the process. We refer the reader to Section 1 of \cite{KPR} for further information.

\subsubsection{Genealogy}
Multidimensional CSBPs may be endowed with a notion of genealogy, so that in the terminology of Section \ref{sec:po} we may think of the process $(Z(t))_{t \geq 0}$ as a collection of $d$-type populations $\left( \mathcal{Z}(t) \right)_{t \geq 0}$ endowed with an \emph{ancestry relation}, so that for certain pairs of particles $b_1 \in \mathcal{Z}(t_1)$ and $b_2 \in \mathcal{Z}(t_2)$, we write $b_1 \leq b_2$ and we say $b_1$ is an ancestor of $b_2$ and $b_2$ is a descendant of $b_1$. 

The pair $\left( \cup_t \mathcal{Z}(t) , \leq \right)$ has the following properties
\begin{itemize}
\item Ancestry is transitive. That is, if $b_1 \leq b_2$ and $b_2 \leq b_3$ then $b_1 \leq b_3$. Moreover, ancestry is reflexive and antisymmetric, so that $b_1 \leq b_2$ and $b_2 \leq b_1$ occurs if and only if $b_1 = b_2$. In particular, the pair $\left( \cup_t \mathcal{Z}(t) , \leq \right)$ is a partially ordered set.
\item For each pair of times $t_1 < t_2$, and each $b_2 \in \mathcal{Z}(t_2)$, there exists a unique $b_1 \in \mathcal{Z}(t_1)$ such that $b_1 \leq b_2$. 
\end{itemize}

We will always take our CSBPs over a filtered space $(\Omega,\mathcal{F},\left( \mathcal{F}_t )_{t \geq 0 } \right)$ in such a way that each $\mathcal{F}_t$ measures the poset $\left( \cup_{ 0 \leq s \leq t} \mathcal{Z}(s) , \leq \right)$ in addition to $(Z_s)_{0 \leq s \leq t}$. 

In the present article we choose not to delve too deeply into technical issues surrounding the genealogy of CSBPs. The interested reader is directed to Bertoin and Le Gall \cite{BLG}.

\subsection{Labelled forests} \label{sec:forestdef} 
\subsubsection{The set of labelled forests}
When a sample of particles are taken from a $d$-dimensional CSBP at some time $T$, by studying their ancestry at a sequence of earlier times $0 = t_0 < t_1 < \ldots < t_m = T$, we obtain a multidimensional forest with generations $0$ through $m$. We now take a moment to identify the collection of possible forests, borrowing notation from \cite{JP}.

Recall that a graph $(V,E)$ is a tree if it is connected and contains no cycles. A tree is \emph{rooted} if there is a designated \emph{root} vertex $v^{(0)} \in V$. A finite graph $(V,E)$ is a rooted forest if it contains $n$ components $\left( \mathcal{T}_1,\ldots, \mathcal{T}_n \right)$, each of which is a rooted tree. For $1 \leq i \leq n$, let $v_i^{(0)}$ denote the root of the tree $\mathcal{T}_i$. For $\ell \geq 1$,  we write $V_i^{(\ell)}$ for the set of vertices a graph distance $\ell $ from each $v_i^{(0)}$. We note that for distinct pairs $(i, \ell)$, the sets $V_i^{(\ell)}$ are disjoint. We collectively refer to the disjoint union $V^{(\ell)} := \bigcup_{ i= 1}^n  V_i^{(\ell)}$ as generation $\ell$ of the forest. Each vertex $w$ in some $V_i^{(\ell+1)}$ has a unique parent vertex $v \in V_i^{(\ell)}$, and in this case we say $w$ is the child of $v$. We say a vertex is a leaf if it has no children, and internal if it does have children. We note every vertex in generation $\ell + j$ has a unique ancestor in generation $\ell$. 

The following definition encapsculates the set of possible ancestral forests associated with sample $\textbf{b} := (b_{i,j} : 1 \leq i \leq d, 1 \leq j \leq k_i )$ of particles from a $d$-dimensional CSBP across a mesh $0 = t_0 < t_1 < \ldots < t_m = T$. 

\begin{defn}
For $k \in \mathbb{Z}_{ \geq 0}^d$ and $m \geq 1$, let $\mathbb{H}^k(m)$ denote the set of quadruplets $\mathcal{H} = (V,E,\phi,\tau)$ such that 
\begin{itemize}
\item The underlying graph $(V,E)$ is a rooted forest $(\mathcal{T}_1,\ldots,\mathcal{T}_n)$.
\item Every leaf of the forest lies in generation $m$.
\item The function $\phi$ is a bijection between the leaves of $(V,E)$ and $[k]$, where
\begin{align*}
[k] := \left\{ (i,j ) : 1 \leq i \leq d, 1 \leq j \leq k_i \right\}.
\end{align*}
\item The labelling function $\tau:V \to [d]$ is any function labelling the vertices such that $\tau$ \emph{respects} $\phi$ in the sense that whenever $\phi(v) = (i,j)$ for some $j$ we have $\tau(v) = i$. There are no constraints on the values taken by $\tau$ on the internal vertices. We call $\tau(v)$ the type of the vertex $v$.
\end{itemize}
\end{defn}
The forest in Figure \ref{fig:tree1} is an example element of $\mathbb{H}^k(m)$ where $k = (5,5)$ and $m = 3$. 

\subsubsection{The ancestral forest associated with a sample}
Suppose we have a set $b = (b_{i,j} : 1 \leq j \leq k_i )$ of particles within the time $T$ population of the continuous-state branching process, such that for each $(i,j)$, the particle $b_{i,j}$ has type $i$. Now suppose we have a mesh $\mathsf{t} := (t_0,\ldots,t_m)$ of times satisfying $t_0 = 0$ and $t_m = T$. Then we have the following definition:

\begin{defn}
For elements $\textbf{b} = (b_{i,j} : 1 \leq j \leq k_i )$ of $\mathcal{Z}(T)$ such that each $b_{i,j} \in \mathcal{Z}_i(T)$, and a mesh $\mathsf{t} = (t_0,\ldots,t_m)$ of $[0,T]$, the ancestral forest $\mathrm{For}_{\mathsf{t}}( \textbf{b} )$ associated with $b$ is the element of $\mathbb{H}^k(m)$ defined as follows:
\begin{itemize}
\item Each leaf of $\mathrm{For}_{\mathsf{t}}(\textbf{b})$ corresponds to some $b_{i,j}$. The leaf $v$ corresponding to $b_{i,j}$ is labelled $\phi(v) = (i,j)$. 
\item For each $0 \leq l \leq m-1$, each vertex of $V^{(\ell)}$ correspond to a time $t_\ell$ ancestor of at least one $b_{i,j}$. The vertex $v$ corresponding to an ancestor of type $i$ is typed $\tau(v) = i$. 
\item A vertex $v$ in some $V^{(\ell)}$ is connected by an edge to some $w$ to $V^{(l+1)}$ if $v$ is an ancestor of $w$ in the CSBP.
\end{itemize}
\end{defn}

The goal of the distributional approach to the coalescent structure of multidimensional CSBPs is to understand the law of the random forest $\mathrm{For}_{\mathsf{t}}( b_{i,j} )$ when the $(b_{i,j})$ are a uniform $k$-sample from $\mathcal{Z}(T)$. To this end, we use a Poissonization set up.

Suppose we have a filtered probability space $\left(\Omega,\mathcal{F}, \left( \mathcal{F}_t \right)_{t \in [0,T]} , P_x \right)$ carrying a CSBP $(Z_t)_{t \in [0,T]}$ starting from $x \in \mathbb{R}_{ \geq 0}^d$ and run until time $T$. Again, we emphasise that the $\sigma$-algebras $\mathcal{F}_t$ contain all of the information about the genealogy of the process until time $t$. 

Now consider the following two extensions of this probability space:
\begin{itemize}
\item Under the first extension $\mathbb{P}_x^{k,T}$ of $P_x$, conditionally on $\mathcal{Z}(T)$ we take a $k$-sample from the time $T$ population, as in Definition \ref{def:k}. 
\item Under the second extension $\mathbb{Q}_x^{\lambda,T}$ of $P_x$, conditionally on $\mathcal{Z}(T)$ we take a $\lambda$-sample from the time $T$ population, as in Definition \ref{def:lambda}.
\end{itemize}

In the next section we discuss the laws of the random ancestral forests associated with samples taken according to these measures.

\subsection{The laws of random ancestral forests} \label{sec:ancestral}

Given a labelled forest $\mathcal{H}$ in  $\mathbb{H}^k(m)$ and a non-leaf vertex $v$, we define the \emph{outdegree} $\mu(v)$ of $v$ to be the multi-index whose $i^{\text{th}}$ component is given by 
\begin{align*}
\mu(v)_i := \# \{ w \in \mathcal{V} : \text{$w$ is a child of $v$ and $\tau(w) = i$} \}.
\end{align*}
Similarly, we define the \emph{rootdegree} $\rho(\mathcal{H})$ of the forest $\mathcal{H}$ to be the multi-index whose $i^{\text{th}}$ component counts the number of roots of trees in the forest who have type $i$, that is
\begin{align*}
\rho(\mathcal{H})_i := \# \{ 1 \leq z \leq n  : \tau (v_z^0)  = i  \}.
\end{align*}
Now given a forest $\mathcal{H}$ of $\mathbb{H}^k(m)$, a mesh $\mathsf{t} = (t_0,\ldots,t_m)$ of times, and an element $x$ of $\mathbb{R}_{ \geq 0}^d$, we define the energy of the forest by 
\begin{align} \label{eq:energy}
\mathsf{E}_{\mathsf{t},x} \left( \mathcal{H} , \lambda \right) := (-1)^{k - \rho} x^\rho \prod_{ \ell = 0}^{m-1} \prod_{ v \in V^{(\ell)} } D_{\tau(v)}^{\mu(v)}[ u] \left( \Delta t_\ell , u( T-t_{\ell+1} , \lambda) \right),
\end{align}
where $\Delta t_\ell := t_{\ell+1 } - t_\ell$ and $\rho = \rho(H)$ is the rootdegree of $\mathcal{H}$. We clarify that $D_{\tau(v)}^{\mu(v)}[ u]\left( \Delta t_\ell , u( T-t_{\ell+1} , \lambda) \right)$ refers to the $\mu(v)^{\text{th}}$ derivative of the $\tau(v)^{\text{th}}$ component of $u$ with respect to the \emph{second} variable, evaluated in the second variable at $\theta = u( T-t_{\ell+1} , \lambda)$. Although not obvious from the definition, for every $\mathcal{H}$ in $\mathbb{H}_k(m)$, the quantity $\mathsf{E}_{t,x} \left( \mathcal{H} , \lambda \right)$ is non-negative. 

The following theorem states that, conditional on the event $\{ \mathcal{S} = k \}$, the probabilities of certain forests under $\mathbb{Q}_x^{\lambda,T}$ have a natural interpretation in terms of a partition function.

\begin{thm} \label{thm:CSBP1}
Suppose under a probablity measure $\mathbb{Q}_x^{\lambda,T}$ we have a $d$-dimensional CSBP run until time $T$, and we take a $\lambda$-sample $\textbf{b} = \left(b_{i,j} : 1 \leq j \leq \mathcal{S}_i \right)$ of particles from the time $T$ population. Let $\mathsf{t} := (t_0,\ldots,t_m)$ be a mesh of times and let $\mathrm{For}_{\mathsf{t}}(\textbf{b})$ be the ancestral forest of the sample $b$ across $\mathsf{t}$. Then the probability that a $\lambda$-sample has $\{ \S = k \}$, and the resulting ancestral forest $\mathcal{H}$ occuring under $\mathbb{Q}_x^{\lambda,T}$ across the mesh $\mathsf{t}$ is some forest $\mathcal{H}$ is given by 
\begin{align} \label{eq:mesh}
\mathbb{Q}_x^{\lambda,T} \left( \S = k , \mathrm{For}_{\mathsf{t}}(\textbf{b}) = \mathcal{H} \right) =  \frac{ \lambda^k}{k!} e^{ - \langle x, u(T,\lambda) \rangle} \mathsf{E}_{\mathsf{t},x}  \left( \mathcal{H} , \lambda \right). 
\end{align}
Consequently, the conditional probability of an ancestral forest $\mathcal{H}$ occuring  given that the sample has size $k$ is given by  
\begin{align} \label{eq:gibbs}
\Q_x^{\lambda,T} \left( \mathrm{For}_{\mathsf{t}} ( b ) = \mathcal{H} \Big| \mathcal{S} = k \right) = \frac{ \mathsf{E}_{\mathsf{t},x} \left( \mathcal{H} , \lambda \right)  }{ \sum_{ \mathcal{I} \in \mathbb{H}^k(m) } \mathsf{E}_{\mathsf{t},x} \left( \mathcal{I} , \lambda \right) }.  
\end{align}
\end{thm}

We now relate this formula to the $\mathbb{P}_x^{k,T}$ case. In particular, in Section \ref{sec:distributional proof} we use a recent generalisation of Fa\`a di Bruno's formula by Johnston and Prochno \cite{JP} to obtain the partition function of the tree energies. Namely, we show that the partition function $ \sum_{ \mathcal{I} \in \mathbb{H}^k(m) } \mathsf{E} \left( \mathcal{I} , \lambda \right)$ satisfies the relation 
\begin{align*}
 \sum_{ \mathcal{I} \in \mathbb{H}^k(m) } \mathsf{E}_{x , \mathsf{t}} \left( \mathcal{I} , \lambda \right) = \frac{ P_x \left[ Z(T)^k e^{ - \langle \lambda, Z(T) \rangle } \right] }{ P_x \left[ e^{ - \langle \lambda, Z(T) \rangle } \right]}.
\end{align*}
By combining this observation with Theorem \ref{thm:CSBP1} with the continuous Poissonization  Theorem \ref{thm:po c}, we immediately obtain the following result characterising the law of the random labelled forest $\mathcal{H}$ under the uniform measure $\mathbb{P}_x^{k,T}$. Recall the measure $\pi^k(\mathrm{d}\lambda)$ defined in Theorem \ref{thm:po c} and the preceding definition \eqref{solo def}, as well as the definition of the event $\{Z(T) \succ k \} := \{ \text{$Z_i(T) > 0$ for each $i$ such that $k_i > 0$} \}$. 

\begin{thm} \label{thm:CSBP2}
Suppose under a probability measure $\mathbb{P}_x^{k,T}$ we have a $d$-dimensional CSBP run until time $T$, and on the event $\{ Z_T \succ k \}$ we take a $k$-sample $\textbf{b} = \left( b_{i,j} : 1 \leq i \leq d, 1 \leq j \leq \mathcal{S}_i \right)$ of particles from the time $T$ population. Then the probability under $\mathbb{P}_x^{k,T}$ that the ancestral forest of this sample across a mesh $\mathsf{t} = (0,t_1,\ldots,t_{m-1},T)$ is equal to some forest $\mathcal{H}$ is given by 
\begin{align*}
\mathbb{P}_x^{k,T} \left( \mathrm{For}_{\mathsf{t}}( \textbf{b} ) = \mathcal{H} , \mathcal{Z}(T) \succ k \right) = \int_{\mathbb{R}_{ \geq 0}^d } \pi^k( \mathrm{d} \lambda) \frac{\lambda^k}{k!} e^{ - \langle x, u(T,\lambda) \rangle} \mathsf{E}_{\mathsf{t},x} \left( \mathcal{H} , \lambda \right).
\end{align*}

\end{thm}
\begin{proof}
Theorem \ref{thm:CSBP2} is an immediate consequence of combining \eqref{eq:mesh} of Theorem \ref{thm:CSBP1} with the continuous Poissonization result, Theorem \ref{thm:po c}.
\end{proof}

\subsection{The special case $m=1$}
We now close this section by studying the special case $m =1$, which amounts to looking at the multidimensional partition induced by the time $0$ ancestors of a sample take at time $T$ (without regard for the ancestral behaviour in the interim period $(0,T)$). When $\mathsf{t} = (t_0,t_1) = (0,T)$, the forest energies take the simple form
\begin{align} \label{eq:special1}
\mathsf{E}_{(0,T),x} ( \mathcal{H} , \lambda ) := (-1)^{k- \rho} x^\rho \prod_{ v \in V^{(0)} } D_{ \tau(v) }^{\mu(v)}[u] (T, \lambda),
\end{align}
where $V^{(0)}$ is simply the set of roots of trees in the forest $\mathcal{H}$. In particular, in the case $ m =1$, Theorem \ref{thm:CSBP2} reads as saying
\begin{align} \label{eq:m1}
\mathbb{P}_x^{k,T} \left( \mathrm{For}_{(0,t)}(\textbf{b}) = \mathcal{H} , \mathcal{Z}(T) \succ k \right) = (-1)^{k- \rho} x^\rho \int_{\mathbb{R}_{ \geq 0}^d } \pi^k( \mathrm{d} \lambda) \frac{ \lambda^k}{k!} e^{ - \langle x, u(T,\lambda) \rangle}   \prod_{ v \in V^{(0)} } D_{ \tau(v) }^{\mu(v)}[u] (T, \lambda).
\end{align}
Writing $t$ in place of $T$, the small-$t$ asymptotics of the formula \eqref{eq:m1} form the starting point of the subject of the next section, in which we study the local coalescent structure of multidimensional continuous-state branching process.

\section{The local coalescent structure of CSBPs} \label{sec:local}

\subsection{Multitype $\Lambda$-coalescents}

We saw in Section 1 that the coalescent structure of one-dimensional CSBPs may be described in terms of the $\Lambda$-coalescents. With a view here to developing an analogous correspondence for higher dimensional CSBPs with coalescent processes we require a notion of a multidimensional coalescent process, as introduced in work in preparation by the first author with Kyprianou and Rogers \cite{JKR}. Short of entering a full discussion of such processes here, we will nonetheless provide a brief outline.

A multidimensional coalescent process is a coalescent process taking values in the set of multidimensional partitions of a set $[k] := \{ (i,j) : 1 \leq i \leq d, 1 \leq j \leq k_i \}$ with the property that each block of the partition itself has a `type' in $\{1,\ldots,d\}$. Briefly, the authors of \cite{JKR} prove a multidimensional analogue of the main result in Pitman \cite{pit}, this multidimensional analogue stating that all `multidimensional exchangeable' coalescent processes with Markovian projections are constructed via the multitype $\Lambda$-coalescents. That is, they are governed by the rule that when the process has $k$ blocks (which is shorthand for saying that for each $i$ there are $k_i$ blocks of type $i$), any $\alpha$ of these blocks are merging to form a single block of type $c$ at a rate 
\begin{align*}
\lambda_{ \alpha, k}^{(c)} := \sum_{ j = 1, j \neq c}^d \ind_{ \alpha = \mathbf{e}_j } \kappa_{c,j}+ \ind_{ \alpha = 2 \mathbf{e}_c} 2 \beta_c + \int_{[0,1]^d} s^\alpha (1-s)^{k-\alpha} Q_c ( \mathrm{d} s),
\end{align*}
where $\kappa_{c,j}$ and $\beta_c$ are non-negative real numbers, and $Q_c$ is a fixed measure. In other words, any exchangeable multidimensional coalescent process with Markovian projections may experience three types of events: single blocks changing type, two blocks of a certain type merging to form a single block of that same type, or large events entailing a random proportion of all extant blocks in the process merging to form a single block of some type.

With the multitype $\Lambda$-coalescents now defined, we will see in  Section 4.3 that the coalescent structures of multidimensional CSBPs may be described in terms of multidimensional $\Lambda$-coalescents, where the merger parameters $\kappa_{c,j}, \beta_c, Q_c$ may be extracted from the characteristic exponent of the CSBP.

\subsection{Special forests in $\mathbb{H}^k(1)$}
Suppose for a small time $t$, under a probability measure $\mathbb{P}_x^{k,t}$ we take a $k$-sample from the time $t$ population of a multidimensional CSBP and study its ancestral forest across the simple mesh $\mathsf{t} = (0,t)$. Since $t$ is small, we expect that it is unlikely that any coalescent events have happened in the small time window $[0,t]$, and hence it is highly likely that each particle sampled at the small time $t$ is descended from an ancestor of the same type at time $0$, and that the ancestors of the different particles are distinct. In other words, we expect that 
\begin{align} \label{eq:trivial}
\lim_{t \to 0} \mathbb{P}_x^{k,t} \left( \mathrm{For}_{(0,t)}(\textbf{b}) = \mathcal{I}^k , \mathcal{Z}(t) \succ k \right) = 1,
\end{align}
where $\mathcal{I}^k$ is the trivial forest in $\mathbb{H}^k(1)$ in which each tree is a \emph{stick}: i.e. each tree has a single leaf $(i,j) \in [k]$, and the root of each tree has the same type as the leaf.

\begin{figure}[ht]

\begin{forest}
[, phantom, s sep = 1cm
[$1$,fill=col1,
[$b_{1,1}$,fill=col1,]
]
[$2$,fill=col2,
[$b_{2,1}$,fill=col2,]
]
[$2$,fill=col2,
[$b_{2,2}$,fill=col2,]
]
[$2$,fill=col2,
[$b_{2,3}$,fill=col2,]
]
[$3$,fill=col3,name=a,
[$b_{3,1}$,fill=col3,name=b,]
]
]
\end{forest}
  \caption{The trivial forest $\mathcal{I}^k$ for $k = (1,3,1)$.
 }\label{fig:trivial tree}
\end{figure}
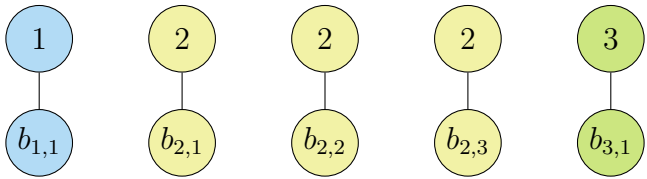
In fact, we will see in Section \ref{sec:smalltime} that given a general forest $\mathcal{H}$ in $\mathbb{H}_k(1)$, for small $t$ we have the asymptotics  
\begin{align} \label{eq:trivial}
\mathbb{P}_x^{k,t} \left( \mathrm{For}_{(0,t)}(\textbf{b}) = \mathcal{H} , \mathcal{Z}(t) \succ k \right) \text{ is of the order }  t^{ \# \{ \text{non-stick trees in $\mathcal{H}$} \} }.
\end{align}
In order to study the rates of the processes, we therefore concern ourselves with the forests in which exactly one tree is not a stick. We define the following special forest with one-non stick tree:

\begin{defn} \label{def:special}
For $\alpha \neq \mathbf{e}_c$, let $\mathcal{H}^k_{c, \alpha}$ be the forest in $\mathbb{H}^k(1)$ with a single non-stick tree of root type $c$ and leaves $\{ (i,j ) : 1 \leq i \leq d , 1 \leq j \leq \alpha_i \}$. 
\end{defn}

For an example of a tree of the form $\mathcal{H}^k_{c,\alpha}$, the reader is invited to consult Figure \ref{fig:tree2}, which displays the forest $\mathcal{H}^k_{c,\alpha}$ in the case $c = 3, \alpha = (3,1,1), k = (4,2,2)$. 

With the $\mathcal{H}^k_{c,\alpha}$ forests now defined, our results on the local coalescent structure of CSBPs amount to identifying the asymptotics  
\[
\lim_{t \to 0} \frac{1}{t} \mathbb{P}_x^{k,t} \left( \mathrm{For}_{(0,t)}(\textbf{b}) = \mathcal{H}^k_{c,\alpha} , \mathcal{Z}(t) \succ k \right).\] 

\subsection{Local coalescence in CSBPs}

The merger measures in the $\Lambda$-coalescents are defined in terms of a pushforward of the L\'evy measures associated with the underlying CSBP. To this end, recall from Section \ref{sec:csbpintro} that every continuous-state branching process is characterised by a branching mechanism $\psi: \mathbb{R}^d \to \mathbb{R}^d$, which is expressed in terms of a L\'evy-Khintchine integral representation involving drift coefficients $(\kappa_{c,j})_{1 \leq c,j \leq d}$, Brownian coefficients $(\beta_c)_{1 \leq c \leq d}$, and L\'evy measures $(\nu_c)_{1 \leq c \leq d}$, where each $\nu_c$ is a measure on $\mathbb{R}_{ \geq 0}^d$ satisfying some integrability conditions . 

Given $x \in \mathbb{R}_{ \geq 0}^d$, define the transformation $T_x: \mathbb{R}_{ >0}^d \to [0,1)^d$ by
\begin{align*}
T_x(r)_j = \frac{r_j}{r_j + x_j},
\end{align*}
and let $T^\#_x  \nu_c $ be the pushforward measure on $[0,1)^d$ of $\nu_c $ under $T_x$. That is, for a Borel subset $A$ of $[0,1)^d$, $T^\#_x  \nu_c (A) := \nu_c  \left( T_x^{-1} (A) \right)$.

We are now ready to state our main result on the local coalescent structure of CSBPs. 

\begin{thm} \label{thm:rate}
Suppose under $\mathbb{P}_x^{k,t}$, we take a $k$-sample from a CSBP at time $t$. Then for a non-zero multi-index $\alpha$ not equal to $\mathbf{e}_c$, the rate at which $\alpha$ particles of the $k$-sample are recently descended from a particle of type $c$ is given by 
\begin{align} \label{eq:local new}
&\lim_{t \to 0} \frac{1}{t} \mathbb{P}_x^{k,t} \left( \mathrm{For}_{(0,t)}(\textbf{b}) = \mathcal{H}^k_{c,\alpha} , \mathcal{Z}(t) \succ k \right)\\
& = \sum_{ j = 1, j \neq c}^d \ind_{ \alpha = \mathbf{e}_j} \kappa_{c,j} \frac{x_c}{x_j} +  \ind_{ \alpha = 2 \mathbf{e}_c} \frac{ 2 \beta_c }{x_c} + x_c \int_{ [0,1]^d} s^\alpha ( 1- s)^{ k- \alpha} T_x^{\#} \nu_c( \mathrm{d} s). \nonumber
\end{align} 
\end{thm}
A few remarks interpreting Theorem \ref{thm:rate} are in order. Namely, equation \eqref{eq:local new} states that if the population size is given by $x \in \mathbb{R}_{ \geq 0}^d$ at time $t$, we observe the following events backwards in time during the small time window $[t,t+\mathrm{d}t)$:
\begin{itemize}
\item A single block of type $j$ changes to type $c$ at rate $\kappa_{c,j} \frac{x_c}{x_j}$. Let us highlight that the proportion of type $j$ particles that are recently descended from type $c$ particles is proportional to the ratio $\frac{x_c}{x_j}$ of the two populations. That is, if the type $c$ population is large compared to the type $j$ population at a certain point in time, then many of the type $j$ particles are recent descendents of type $c$ particles.
\item Two blocks of type $c$ merge to form a single block of type $c$ at rate $\frac{ 2 \beta_c }{x_c}$, which are precisely the dynamics observed by Donnelly and Kurtz relating the one-dimensional Feller diffusion to the one-dimensional Kingman coalescent \cite{DK}.
\item Large mergers of various blocks of various types merging to form a single block of type $c$ are governed by the $\Lambda$-coalescent type dynamics through the pushforward measure $T_x^{ \# } \nu_c$. 
\end{itemize}
We now take a moment to relate Theorem \ref{thm:rate} to one-dimensional CSBPs. Here, Theorem \ref{thm:rate} states that
\begin{align} \label{eq:manfred}
&\lim_{t \to 0} \frac{1}{t} \mathbb{P}_x^{k,t}\left( \pi_t = \left\{ \{1,\ldots,j \}, \{j+1\}, \ldots, \{k \} \right\} , Z(t) > 0 \right) \nonumber \\
&= \ind_{ j = 2} \frac{ 2 \beta}{ x} +x \int_{[0,1)} s^{ j } (1-s)^{k - j} T_x^{\#} \nu( \mathrm{d} s).
\end{align}
It is a straightforward calculation that Theorem 1.2 of the introduction is a consequence of neatly packing \eqref{eq:manfred} into a single integral by setting $\Lambda^\psi( x, \mathrm{d}s) = \frac{2 \beta}{ x} \delta_0( \mathrm{d} s ) + x s^2 T_x^{ \# } \nu( \mathrm{d} s)$. 

Finally, we conclude this section by taking a moment to discuss the application of Theorem \ref{thm:rate} to the simpliest higher dimensional CSBP, the multidimensional Feller diffusion, the $d$-dimensional continuous state branching process with $\nu_c = 0$ for each $c \in \{1,\ldots , d\}$, and solution to the stochastic differential equation \eqref{feller SDE}. The components of the branching mechanism of Feller diffusion are given by 
 \begin{align} \label{eq:feller mech}
\psi_c(\la) = - \sum_{ j = 1}^d \kappa_{c,j} \lambda_j + \beta_c \lambda_c^2
\end{align}
In particular, by plugging $\nu_c = 0$ into Theorem \ref{thm:rate}, we recover the following result, which may be regarded as a multidimensional analogue of Theorem 5.1 of Donnelly and Kurtz \cite{DK}.
\begin{cor} \label{cor:DK}
Under a probability measure $\mathbb{P}_x^{k,t}$, suppose we have a Feller diffusion $(Z_u)_{0 \leq u \leq t}$ starting from $x \in \mathbb{R}_{\geq 0}^d$ and with branching mechanism \eqref{eq:feller mech}. Suppose first we sample two particles $b_{c,1}$ and $b_{c,2}$ uniformly and independently from the type $c$ population at time $t$, and let $A(c,c)$ be the event that the $b_{c,1}$ and $b_{c,2}$ are descended from the same ancestor of type $c$ in the time zero population. Then
\begin{align*}
\lim_{t \to 0} \frac{1}{t} \mathbb{P}_x^{k,t} \left( A(c,c) , Z_c(t) > 0 \right) = \frac{2 \beta_c}{ x_c}.
\end{align*}
Alternatively, again under $\mathbb{P}_x^{k,t}$, sample a single particle $b_{j,1}$ of type $j$ from the type $j$ population at time $t$, and let $B(c,j)$ be the event that $b_{j,1}$ is descended from a particle of type $c$ in the time zero population. Then
\begin{align*}
\lim_{t \to 0} \frac{1}{t} \mathbb{P}_x^{k,t} \left( B(c,j) , Z_j(t) > 0 \right) = \frac{x_c}{x_j } \kappa_{c,j}.
\end{align*}
\end{cor}

That completes the section on the local coalescent structure of multidimensional CSBPs. The remainder of the paper is dedicated to proving the results stated so far. We begin in Section \ref{sec:po proof} by proving the results stated in Section \ref{sec:po}. 
\section{Proofs of the Poissonization theorems} \label{sec:po proof}
In this section we prove our main Poissonization results, Theorem \ref{thm:po c} and Theorem \ref{thm:po d}. Throughout this section, for elements $\alpha, \beta$ of $\mathbb{R}_{ \geq 0}^d$ we will use the notation $\alpha \succ \beta$ to denote that $\alpha_i > 0$ for each $i$ such that $\beta_i > 0$. 

\subsection{Continuous Poissonization: Proof} \label{sec:cproofs}
\bp[Proof of Theorem \ref{thm:po c}]
Recall the definitions introduced in Section \ref{sec:cpoi}. Recall that we have the product sigma algebra $\tF := \mathcal{F} \otimes \mathcal{F}'$. We note that $\mathcal{F}$ may be associated with a subalgebra $\mathcal{F}_0$ of the product algebra $\tF := \F \otimes \F'$ by setting 
\begin{align*}
\mathcal{F}_0 := \{ B \times \Omega' : B \in \mathcal{F} \}. 
\end{align*}
For measures $\rho$ on $\mathbb{R}^d_{\geq 0}$, and $k \in \mathbb{Z}_{\geq 0}^d$, define the measure $R_{\rho,k}$ on $\tF$ by
\begin{align*}
R_{\rho,k}(A) := \int_{\mathbb{R}_{\geq 0}^d} \rho( \mathrm{d} \lambda) \Q^{\lambda}(A, \mathcal{S} = k ), ~~~ A \in \tF. 
\end{align*}
Theorem \ref{thm:po c} is equivalent to the statement that  
\begin{align} \label{super}
R_{\rho,k}(A) = \P^k(A, Z \succ k) 
\end{align}
holds for every $A$ in $\tF$ if and only if 
\begin{align} \label{keymixer} 
\rho( \mathrm{d} \lambda ) =  \pi^k( \mathrm{d} \lambda).
\end{align}
The remainder of the proof is structured as follows. First we prove that if $\rho(  \mathrm{d} \lambda) = \pi^k ( \mathrm{d} \lambda)$, then \eqref{super} holds for sets $A$ in the subalgebra $\F_0$ of $\tF$. Then we use properties of Poisson processes to show that when $\rho( \mathrm{d} \lambda) = \pi^k( \mathrm{d} \lambda) $, \eqref{super} holds on all $A$ in $\mathcal{F}$, completing one direction of the proof. We then use a Laplace transform argument to establish the uniqueness of $\pi^k$, proving the other direction.

First we show that with $\rho( \mathrm{d} \lambda) = \pi^k(  \mathrm{d} \lambda)$, \eqref{super} holds for all $A \in \mathcal{F}_0$. To this end, we note that conditionally on $\F_0$,
\begin{align} \label{Q cond}
\Q^\lambda(\mathcal{S} = k | \mathcal{F}_0 ) = \prod_{i=1}^d \frac{( \lambda_i  Z_i )^{k_i}}{k_i!} e^{ - \lambda_i Z_i }.
\end{align}
Now for $A \in \mathcal{F}_0$, using \eqref{Q cond} with the definition of conditional expectation to obtain the second equality below, and the fact that $\Q^\lambda = P$ on $\mathcal{F}_0$ in the third, for any measure $\rho$ on $\mathbb{R}_{ \geq 0}^d$ we have
\begin{align}
R_{\rho,k}[\ind_A]  &:= \int_{\mathbb{R}_{\geq 0}^d}\rho(  \mathrm{d} \lambda) \Q^{\lambda} \left( A,  \mathcal{S} = k \right) \nonumber   \\
&= \int_{\mathbb{R}_{\geq 0}^d}\rho(  \mathrm{d} \lambda) \Q^{\lambda} \left[ \ind_A \prod_{i=1}^d \frac{( \lambda_i  Z_i )^{k_i}}{k_i!} e^{ - \lambda_i Z_i }  \right] \nonumber   \\ 
&= \int_{\mathbb{R}_{\geq 0}^d} \rho(  \mathrm{d} \lambda)  P \left[ \ind_A \prod_{i=1}^d \frac{( \lambda_i  Z_i )^{k_i}}{k_i!} e^{ - \lambda_i Z_i } \right] \nonumber \\
&=  P \left[ \ind_A  \int_{\mathbb{R}_{\geq 0}^d} \rho(  \mathrm{d} \lambda)    \prod_{i=1}^d \frac{( \lambda_i  Z_i )^{k_i}}{k_i!} e^{ - \lambda_i Z_i } \right], \label{mixer}
\end{align}
where in the final equality above we used Fubini's theorem (the integrands are non-negative). Consider plugging $\rho = \pi^k$ into \eqref{mixer}. Indeed, for $z \in \mathbb{R}_{ \geq 0}$, and integer $j$, with $\pi^j( \mathrm{d} \lambda)$ as in \eqref{solo def} we have
\begin{align} \label{eq:gamma id}
\int_{ \mathbb{R}_{ \geq 0 }} \pi^j( \mathrm{d} \lambda) \frac{ ( \lambda z)^j}{ j!} e^{ - \lambda z} = \ind_{ j = 0} + \ind_{ j > 0, z> 0}.
\end{align}
(The case $j = 0$ above is immediate, the case $j > 0$ follows from the gamma integral.) 
It now follows that for $z \in \mathbb{R}_{\geq 0}^d$, using the fact that $\pi^k( \mathrm{d} \lambda) = \pi^{k_1} ( \mathrm{d} \lambda_1) \ldots \pi^{k_d}(  \mathrm{d} \lambda_d)$, we have 
\begin{align}  \label{gamma} 
\int_{\mathbb{R}_{\geq 0}^d} \pi^k(  \mathrm{d} \lambda)    \prod_{i=1}^d \frac{( \lambda_i  z_i )^{k_i}}{k_i!} e^{ - \lambda_i z_i } = \prod_{ i = 1}^d \int_{\mathbb{R}_{\geq 0}}  \pi^{k_i}(  \mathrm{d} \lambda_i) \frac{( \lambda_i  z_i )^{k_i}}{k_i!} e^{ - \lambda_i z_i } = \prod_{ i : k_i > 0 } \ind_{ z_ i > 0 } = \ind_{ Z \succ k },
\end{align}
where $\{ Z \succ k \}$ is given in \eqref{J def}.

Plugging \eqref{gamma} into \eqref{mixer} in the case $\rho = \pi^k$, it follows that for all non-negative $\mathcal{F}_0$ measurable random variables we have 
\begin{align*}
R_{\pi^k,k}(A)  = P \left( A ,  Z \succ k \right),
\end{align*} 
proving that when $\rho= \pi^k$,  \eqref{super} holds for all $A$ in the coarser $\sigma$-algebra $\mathcal{F}_0$.

We now show that with $\rho = \pi^k$, \eqref{super} holds for all $A$ in the richer $\sigma$-algebra $\tF$. Indeed, suppose we take a $\lambda$-sample from a $d$-type population $\mathcal{Z} = \bigcup_{ i = 1}^d \{i \} \times \mathcal{Z}_i$, and let $\mathcal{S}_i$ be the number of points sampled from $\{i\} \times \mathcal{Z}_i$. For each $i$, conditional on the event $\{ \mathcal{S}_i = k_i \}$, these points are uniformly and independently sampled from $\{ i \} \times \mathcal{Z}_i$. Since, for any $\rho$, the measure $R_{\rho,k}$ is supported on the event $\{ \mathcal{S} = k\}$, we have
\begin{align} \label{conditional}
R_{\rho,k}(A | \mathcal{F}_0, \{ \mathcal{S} = k  \} ) =  \P^k(A | \mathcal{F}_0, \{ Z \succ k  \} ).
\end{align}
Now suppose $B$ in $\tF$ such that $B \subseteq \{ Z \succ k  \}$. Now when $\rho = \pi^k$, using the tower property to obtain the first inequality below, \eqref{conditional} as well as the fact that \eqref{super} holds for all $A$ in $\mathcal{F}_0$ to obtain the second, and then the tower property again to obtain the third, and finally the fact that $B$ is a subevent of $\{ Z \succ k \}$ to obtain the fourth, we have
\begin{align*}
R_{\pi^k,k}(B) = R_{\pi^k,k} [ R_{\pi^k,k} (B | \mathcal{F} ) ] =  \P^k[ \P^k(B | \mathcal{F} ) ] = \P^k(B) = \P^k(B, Z \succ k ).  
\end{align*}
proving that \eqref{super} holds for all $B$ in $\tF$. 

It remains to prove the uniqueness of the mixture measure. Suppose for a measure $\rho$ on $\mathbb{R}_{ \geq 0}^d$, \eqref{super} holds for every $A \in \tF$. Then by \eqref{mixer}, for any $\theta \in \mathbb{R}_{ \geq 0}^d$ we must have
\begin{align}
P [ e^{ - \langle \theta, Z \rangle } \ind_{ Z \succ k } ] =P \left[e^{ - \langle \theta, Z \rangle }   \int_{\mathbb{R}_{\geq 0}^d } \rho( \mathrm{d} \lambda)   \prod_{i=1}^d \frac{( \lambda_i  Z_i )^{k_i}}{k_i!} e^{ - \lambda_i Z_i }   \right]~ \text{for every $\theta$ in $\mathbb{R}_{\geq 0}^d$}. \label{theta eq}
\end{align}
Since the span of functions of the form $e^{ - \langle \theta, ~\cdot~ \rangle}$ is dense in $\mathcal{C}_0(\mathbb{R}_{\geq 0}^d)$, \eqref{theta eq} is equivalent to 
\begin{align} \label{z eq}
 \int_{\mathbb{R}_{\geq 0}^d } \rho( \mathrm{d} \lambda)   \prod_{i=1}^d \frac{( \lambda_i  z_i )^{k_i}}{k_i!} e^{ - \lambda_i z_i }  = \ind_{ z \succ k} ~~ \text{for every $z $ in $ \mathbb{R}_{\geq 0}^d$}. 
\end{align}
In particular \eqref{z eq} determines the Laplace transform of the measure $\hat{\rho}( \mathrm{d} \lambda) := \prod_{i=1}^d \lambda_i^{k_i} \rho( \mathrm{d} \lambda)$, and hence determines $\rho$.
\ep

We now work towards a proof of Theorem \ref{thm:c}. We start with a lemma concerning Poissonized probability distributions on $\mathbb{R}_{ \geq 0}^d$. Recall that for $k = (k_1,\ldots,k_d)$ in $\mathbb{Z}_{ \geq 0}^d$ and $x \in \mathbb{R}_{ \geq 0}^d$ we write $x \succ k$ if $x_i > 0$ whenever $k_i > 0$, $k! := k_1! \ldots k_d! $, and whenever $x \succ k$ define
\begin{align*}
x^k := x_1^{k_1}\ldots x_d^{k_d}.
\end{align*}
Finally, for $\lambda, w \in \mathbb{R}_{ \geq 0}^d$ we write $\lambda w$ for the element of $\mathbb{R}_{ \geq 0}^d$ given by $(\lambda w)_i := \lambda_i w_i$. 

\begin{lem}
Let $k = (k_1,\ldots,k_d)$ be a vector of non-negative integers, and under a probability measure $P$ let $Z$ be an $\mathbb{R}_{ \geq 0}^d$-valued random variable satisfying $\mathbb{P} \left( Z \succ k   \right) > 0$. Let $\pi^k( \mathrm{d} \lambda)$ be as in Theorem \ref{thm:po c}. Then
\begin{align*}
\Pi^k( \mathrm{d} \lambda) := P\left[ \frac{ (\lambda Z)^k }{ k!} e^{ - \langle \lambda, Z \rangle} \Big| Z \succ k  \right]  \pi^k(  \mathrm{d} \lambda)  
\end{align*}
is a probability measure on $\mathbb{R}_{ \geq 0}^d$.
\end{lem}
\begin{proof}
The density $\Pi^k( \mathrm{d} \lambda)$ is clearly non-negative on $\mathbb{R}_{ \geq 0}^d$, so it remains to show it has unit mass. Using Fubini's theorem to obtain the second equality below, 
\begin{align*}
\int_{\mathbb{R}_{ \geq 0}^d } \Pi^k( \mathrm{d} \lambda) &:= \int_{ \mathbb{R}_{ \geq 0 }^d } P\left[ \frac{ (\lambda Z)^k }{ k!} e^{ - \langle \lambda, Z \rangle} \Big| Z \succ k  \right]  \pi^k(  \mathrm{d} \lambda)   \\
 &=   P\left[  \int_{ \mathbb{R}_{ \geq 0 }^d }\frac{ (\lambda Z)^k }{ k!} e^{ - \langle \lambda, Z \rangle}   \pi^k(  \mathrm{d} \lambda)  \Big| Z \succ k  \right]  \\
&=    P\left[  \prod_{ i =1}^d \int_{\mathbb{R}_{ \geq 0}^d} \pi^{k_i}( \mathrm{d} \lambda_i) \frac{ (\lambda_i Z_i )^{k_i} }{k_i!}  e^{ - \lambda_i Z_i }  \Big| Z \succ k  \right]  .
\end{align*}
The proof now follows from using \eqref{eq:gamma id}.
\end{proof}

We are now ready to prove Theorem \ref{thm:c}. Recall from the proof of Theorem \ref{thm:po c} the definition of $\mathcal{F}_0$.

\begin{proof}[Proof of Theorem \ref{thm:c}]
Consider computing $\Q^\lambda( \S = k )$. Using the fact that $\Q^{\lambda} = P$ on the subalgebra $\mathcal{F}_0$ of $\mathcal{F} \otimes \mathcal{F}'$ in the second equality below, and the fact that conditionally on $Z := \mathrm{Leb}(\mathcal{Z}_i)$, the variables $\S_i$ are conditionally Poisson distributed with parameter $\lambda_i Z_i$ to obtain the third, we have 
\begin{align*}
\Q^{\lambda} \left( \mathcal{S} = k \right) &= \Q^{\lambda} \left[ \Q^\lambda \left( \mathcal{S} = k \big| Z_1,\ldots,Z_d \right) \right]\\
&= P \left[ \Q^\lambda \left( \mathcal{S} = k \big| Z_1,\ldots,Z_d \right) \right]\\
&= P \left[ \prod_{ i = 1}^d \frac{ ( \lambda_i Z_i )^{k_i} }{ k_i!} e^{ - \lambda_i Z_i } \right] =  P\left[ \frac{ (\lambda Z)^k }{ k!} e^{ - \langle \lambda, Z \rangle} \right]  .
\end{align*}
Now expand
\begin{align*}
\Q^\lambda \left( A, \S = k  \right)  = \Q^\lambda \left( A \big| \S = k \right) \Q^\lambda \left( \S = k \right)
\end{align*}
in \eqref{c multi eq}, and divide through by $\mathbb{P} \left( Z \succ k  \right)$. 
\end{proof}

\subsection{Discrete Poissonization: Proof}
Since the proof is basically identical to the continuous case -- with the beta integral in the discrete case playing the role of the gamma integral in the continuous case -- we content ourselves with simply outlining the key steps here. 
\bp[Proof of Theorem \ref{thm:po d}]
Recall the definitions from Section \ref{disc poiss}. Let $\rho$ be a measure on $[0,1)^d$ and define the measure $R_{\rho,k}$ on $\tF$ by
\begin{align*}
R_{\rho,k}(A) := \int_{[0,1)^d} \rho(dp) \Q^{p}(A, \mathcal{S} = k ), ~~~ A \in \tF. 
\end{align*}
Our aim is to show that $R_{\bar{\pi}^k,k}(A) = \P^k(A, N \geq k)$ for every $A \in \tF$, and that $\bar{\pi}^k$ is the unique measure with this property.

Again, first we show that $R_{\bar{\pi}^k,k}(A) = \P^k(A)$ for every $A$ in the coarser $\sigma$-algebra $\mathcal{F}_0$. To this end, note that for $p \in [0,1)^d$, $\Q^p( \mathcal{S} = k | \mathcal{F}_0) = \prod_{i=1}^d \ind_{N_i \geq k_i} \binom{N_i}{k_i} p_i^{k_i} (1 - p_i)^{N_i - k_i}$, and hence for any $A \in \mathcal{F}_0$ and any measure $\rho$ on $[0,1)^d$, we have
\begin{align} \label{k prod}
R_{\rho,k} (A)  := P \left[ \ind_A \int_{[0,1)^d} \rho(d p)  \prod_{i=1}^d \ind_{N_i \geq k_i}  \binom{N_i}{k_i} p_i^{k_i} (1 - p_i)^{N_i - k_i} \right] .
\end{align}
Plugging in $\rho(dp) = \bar{\pi}^k(dp)$ in \eqref{k prod}, and using the beta integral $\int_0^1 dp ~ p^x(1-p)^y = \frac{ x! y!}{ (x+ y + 1)!}$ we obtain
\begin{align} \label{k prod 2}
R_{\bar{\pi}^k,k} (A)  := P \left( A , N \geq k \right) ~ \text{for all $A \in \mathcal{F}_0$}. 
\end{align}
Finally, note that given that a labelled Bernoulli process has chosen $k$ points in a set $H$, these $k$ points are sampled uniformly and without replacement from $H$. Since $R^k_{\bar{\pi}^k}$ is supported on $\{\mathcal{S} = k\}$, it follows that
\begin{align} \label{k prod 3}
R_{ \bar{\pi}^k,k}(  A | \mathcal{F}_0) = \P^k ( A | \mathcal{F}_0) ~ \text{for all $A$ in $\mathcal{F}_0$}. 
\end{align}
Combining \eqref{k prod 2} and \eqref{k prod 3} we show that \eqref{d po eq} holds for all $A$ in $\tF$, which completes one direction of the proof.

The uniqueness of $\bar{\pi}^k(dp)$ follows from a probability generating function argument similar to the one used in the continuous case. 
\ep

\section{Proof of Theorem \ref{thm:CSBP1}} \label{sec:distributional proof}

In this section we prove Theorem \ref{thm:CSBP1} concerning the law of the random forests $\mathrm{For}_{\mathsf{t}}(\textbf{b})$ occuring when taking $\lambda$-samples from multidimensional continuous-state branching process. 
\subsection{Uniform and Poisson sampling from multidimensional continuous state branching processes}
Suppose under a probability space $(\Omega,\mathcal{F}_T,P_x)$, we have a $d$-dimensional continuous-state branching process $(Z(t))_{0 \leq t \leq T}$ starting from $x \in \mathbb{R}_{\geq 0}^d$ and with branching mechanism $\psi:\mathbb{R}_{\geq 0}^d \to \mathbb{R}_{ \geq 0}^d$. We recall from Section \ref{sec:csbpintro} that 
\begin{align} \label{branch exp 2}
P_x \left[ e^{ - \langle \lambda, Z(t) \rangle } \right] = e^{ - \langle x , u(t,\lambda) \rangle } 
\end{align}
where $u(T,\lambda)$ is the solution to the partial differential equation
\begin{align*}
\frac{ \partial }{ \partial t} u(t,\lambda) = - \psi \left( u(t,\lambda) \right).
\end{align*}
Recall the probability measure $\mathbb{Q}_x^{\lambda,T}$ extending $P_x$ to a richer measurable space $(\tO, \tF_T)$ carrying random variables $\mathcal{S}$ and $b$, where $\mathcal{S}$ takes values in $\mathbb{Z}_{ \geq 0}^d$, and $ \textbf{b} = \left(b_{i,j} : 1 \leq i \leq d , 1 \leq j \leq \mathcal{S}_i\right)$ is a $\lambda$-sample on $\mathcal{Z}(T)$: a random array such that each $b_{i,j}$ is an element of $\mathcal{Z}_i(T)$.

\begin{notation}
We would like to make a quick remark about notation. The event $\{ \mathrm{For}_{\mathsf{t}}(\textbf{b}) = \mathcal{H} \}$ is a subevent of $\{ \S = k \}$, which itself is a subevent of $\{ \mathcal{Z}(T) \succ k \}$. In particular, in the remainder of the paper we will write $\mathbb{Q}_x^{\lambda,T} \left( \mathrm{For}_{(0,t)}(\textbf{b})= \mathcal{H} \right) = \mathbb{Q}_x^{\lambda,T} \left( \mathrm{For}_{(0,t)}(\textbf{b})= \mathcal{H} , \S = k \right) $ and $\mathbb{P}_x^{k,T} \left( \mathrm{For}_{(0,t)}(\textbf{b})= \mathcal{H} \right) = \mathbb{P}_x^{ k ,T} \left( \mathrm{For}_{(0,t)}(\textbf{b})= \mathcal{H} , \mathcal{Z}(T) \succ k \right) $.
\end{notation}

\subsection{Multi-type Poisson processes and multidimensional CSBPs}
Whenever $\mathcal{W}$ is a subset of $\mathcal{Z}(t)$, we write $\Leb(\mathcal{W})$ for the element of $\mathbb{R}_{\geq 0}^d$ where $\Leb(\mathcal{W})_i = \Leb( \mathcal{W} \cap \mathcal{Z}_i(t) )$. In particular, $\Leb(\mathcal{Z}(0)) = x$ almost-surely under $\Q_x^{\lambda,T}$. Now consider a subset $\mathcal{W}$ of $\mathcal{Z}(0)$. The outdegree of $\mathcal{W}$, written $\mu(\mathcal{W})$, is the $\mathbb{Z}_{\geq 0 }^d$-valued and $\tF_T$-measurable random variable whose $i^{\text{th}}$ component counts the number of type-$i$ descendants of particles in $\mathcal{W}$ in the $\lambda$-sample. That is
\begin{align*}
\mu(\mathcal{W})_i := \# \{ j : \text{$b_{i,j}$ is a descendant of a particle in $\mathcal{W}$}\}.
\end{align*}
The following observation is the central idea of this section, and is an immediate consequence of the branching property at time $0$ together with the Poissonian structure of sampling at time $T$.
\begin{quote} \label{quopte}
Let $\mathcal{W}$ and $\mathcal{W}'$ be disjoint subsets of $\Zc(0)$. Then under $\Q^{\lambda,T}$, the outdegrees $\mu(\mathcal{W})$ and $\mu(\mathcal{W}')$ are independent. 
\end{quote}
In fact, this observation implies the following Corollary. (For further details, we refer the reader to Chapter 3 of \cite{kallenberg}, in particular Theorem 3.17.) 
\begin{cor} \label{cor:poisson}
For $1 \leq i \leq d$ and a non-zero multi-index $\alpha$, particles in $\Zc_i(0)$ who have outdegree $\alpha \in \mathbb{Z}_{\geq 0}^d$ occur in $\Zc_i(0)$ according to a Poisson process of some rate $r_i^\alpha(T,\lambda)$. Moreover, these Poisson processes are independent for different values of $(i,\alpha)$.
\end{cor}
We now turn our discussion to identifying the rates $\{ r_i^\alpha(T,\lambda) : 1 \leq i \leq d, \alpha \in \mathbb{Z}^d_{\geq 0} \text{ nonzero} \}$. To this end, on many occasions in the sequel we will like to interchange the order of differentiation and expectation. The following technical lemma justifies these interchanges. Recall that we write $\lambda \succ \alpha$ if $\lambda_i > 0$ for each $i$ such that $\alpha_i > 0$. We note that $h^\alpha(r) := r^\alpha e^{ - \langle \lambda, r \rangle}$ is a bounded function for $r \in \mathbb{R}_{ \geq 0}^d$ if and only if $\lambda \succ \alpha$. 

\begin{lem} \label{techlem}
Let $\lambda \in \mathbb{R}_{ \geq 0}^d$ and let $J(\mathrm{d}r)$ be a finite measure on $\mathbb{R}_{ \geq 0 }^d$. Then whenever $\lambda \succ \alpha$ we have 
\begin{align*}
D^\alpha \int_{ \mathbb{R}_{ \geq 0}^d } J(\mathrm{d}r) e^{ - \langle \lambda, r \rangle }=   (-1)^{|\alpha|} \int_{ \mathbb{R}_{ \geq 0}^d } J(\mathrm{d}r)   r^\alpha   e^{ - \langle \lambda, r \rangle } .
\end{align*}
In particular, by letting $J$ be the law of $Z(T)$, we see that we may interchange the order of differentiation and expectation in the Laplace transform of the process $(Z(t))_{t \geq 0}$, so that
\begin{align} \label{eq techlem 2}
P_w \left[ Z(T)^\alpha e^{ - \langle \lambda, Z(T) \rangle } \right] = (-1)^{|\alpha|} D^\alpha e^{ - \langle w , u(T,\lambda) \rangle } .
\end{align}
\end{lem}
\bp
We proceed by induction on the multi-index $\alpha$. The result is clearly true for $\alpha = (0,\ldots,0)$, so it remains to prove an inductive step. Suppose the result holds for some multi-index $\beta \prec \lambda$, then we show the result also holds for any multi-index $\beta + \mathbf{e}_i \prec \lambda$. To see this, note that
\begin{align*}
- r^{ \beta + \mathbf{e}_i } e^{ - \langle \lambda , r \rangle } = D^{\mathbf{e}_i} r^\beta e^{ - \langle \lambda, r \rangle } = \lim_{ h \to 0} f_h(r),
\end{align*}
where $f_h(r) := r^\beta e^{ - \langle \lambda , r \rangle } \left( \frac{ e^{ - h r_i } - 1 }{h} \right)$.  
Now for every $q > 0$ we have $\frac{ 1 - e^{ - q}}{q} < 1$, and in particular for $h > 0$, $|f_h(r)| \leq  g(r) := r^{ \beta + \mathbf{e}_i}  e^{ - \langle \lambda,  r \rangle }$. Since $\beta + \mathbf{e}_i \prec \lambda$, $g(r)$ is a bounded function of $r$. It follows by the bounded convergence theorem that
\begin{align*}
\int_{ \mathbb{R}_{ \geq 0}^d} J(\mathrm{d}r) \lim_{h \to 0} f_h(r) = \lim_{h \to 0}\int_{ \mathbb{R}_{ \geq 0}^d} J(\mathrm{d}r)  f_h(r) 
\end{align*}
which amounts to the statement of the lemma. 
\ep

The following lemma gives the probability that a subset of the initial population has a certain outdegree. 
\begin{lem} \label{christmas lemma}
Let $\mathcal{W}$ be a measurable subset of $\mathcal{Z}(0)$, and suppose $\Leb(\mathcal{W}) = w \in \mathbb{R}_{ \geq 0}^d$. Then for $\alpha \in \mathbb{Z}_{ \geq 0}^d$
\begin{align} \label{christmas2}
\Q^{\lambda,T}_x \left( \mu(\mathcal{W}) = \alpha | \Leb(\mathcal{W}) = w \right) = \frac{(-1)^{|\alpha|} \lambda^\alpha }{ \alpha!} D^\alpha \exp \left( - \langle w, u(T,\lambda) \rangle \right).
\end{align}
\end{lem}
 
\bp
Clearly $\Q^{\lambda,T}_x \left( \mu(\mathcal{W}) = \alpha | \Leb(\mathcal{W}) = w \right) $ may only be non-zero if $\alpha \succ \lambda$, for if $\lambda_i= 0$ then no particle of type $i$ is taken in the $\lambda$-sample of $\mathcal{Z}(T)$. Now let $\mathcal{W}_T \subseteq \mathcal{Z}(T)$ be the set of time $T$ descendants of particles in $\mathcal{W}$. Then plainly $\Leb(\mathcal{W}_T)$ has the same law as $Z(T)$ under $P_w$. It follows that for $\alpha \succ \lambda$, using the probabilities of the Poisson distribution we have
\begin{align}
\Q^{\lambda,T}_x \left( \mu(\mathcal{W}) = \alpha | \Leb(\mathcal{W}) = w \right) &= P_w \left[ \prod_{ i = 1}^d \frac{ \left( \lambda_i Z_i(T) \right)^{\alpha_i}  e^{ - \lambda_i Z_i(T) } }{ \alpha_i! } \right] . \label{christmas} 
\end{align}
The equation \eqref{christmas2} now follows from apply the interchange equation \eqref{eq techlem 2} to \eqref{christmas}.
\ep

The following lemma identifies the rates $r_i^\alpha(T,\lambda)$ described in Corollary \ref{cor:poisson}.

\begin{lem} \label{lem:ancestors}
Let $x \in \mathbb{R}_{\geq 0}^d$. Then
\begin{align} \label{gen rates}
\lim_{ h \to 0 } \frac{1}{h} \Q^{\lambda,T}( \mu(\mathcal{W}) = \alpha | \Leb(\mathcal{W}) = h x)  = \frac{(-1)^{|\alpha|+1}\lambda^\alpha  }{ \alpha!} D^\alpha \langle  x , u(T,\lambda) \rangle.
\end{align}
In particular, by letting $x = \mathbf{e}_i$, we have
\begin{align} \label{outdegree rates}
r_i^\alpha(T,\lambda) := \frac{(-1)^{|\alpha|+1}\lambda^\alpha  }{ \alpha!} D^\alpha u_i (T,\lambda) .
\end{align}
\end{lem}
\bp
Consider setting $w = hx$ in Lemma \ref{christmas lemma}, and consider the quantity $f(h) := \frac{(-1)^{|\alpha|} \lambda^\alpha }{ \alpha!} D^\alpha \exp \left( - \langle h x , u(T,\lambda) \rangle \right)$. We note that since $\alpha$ is non-zero that $f(0) = 0$, and in particular 
\begin{align*}
\lim_{ h \to 0 } \frac{1}{h} \frac{(-1)^{|\alpha|}\lambda^\alpha  }{ \alpha!} D^\alpha \exp \left( - \langle h x , u(T,\lambda) \rangle \right)  = \lim_{h \to 0} \frac{1}{h} \left( f(h) - f(0) \right) = \frac{ \partial}{ \partial h}\Big|_{h = 0} f(h) .
\end{align*}
The function $g:\mathbb{R}_{ \geq 0} \times \mathbb{R}^d \to \mathbb{R}$ given by $g(h,\lambda) := \exp \left( - \langle h x , u(T,\lambda) \rangle \right)$ is smooth on $\mathbb{R}^{d+1}$, it follows that we may interchange the order of $ \frac{ \partial}{ \partial h}\Big|_{h = 0}$ and $D^\alpha$, thereby obtaining \eqref{gen rates}. 
\ep

We point out here that Foucart, Ma and Mallein \cite{FMM} study the interaction of Poisson processes with one-dimensional CSBPs, with the one-dimensional case of the formula \eqref{outdegree rates} appearing in Lemma 5.8 of that paper.

For fixed $i$, the sum of the rates $\sum_{ \alpha \neq 0} r_i^\alpha(T, \lambda)$ should characterise the rate of those particles $\Zc_i(0)$ who have \emph{any} positive number of sampled descendants of any type under $\Q_x^{\lambda,T}$. The following lemma characterises this total rate. In order to clarify the possible perspectives on this total rate, we provide two proofs, the first probabilistic and the latter analytic.
\begin{lem} \label{taylor}
Under $\Q_x^{\lambda,T}$, particles in $\Zc_i(0)$ who have at least one sampled descendant of any type at time $T$ occur according to a Poisson process of rate $u_i(T,\lambda)$. In particular, we have the equality 
\begin{align} \label{r eq}
u_i(T,\lambda)  = \sum_{ \alpha \neq 0 } r_i^\alpha(T,\lambda),
\end{align}
where the sum is taken over all non-zero $\alpha$ in $\mathbb{Z}_{\geq 0}^d$.
\end{lem}
\bp[Proof 1: Poisson process calculation]
Let $\Leb(\mathcal{W}) = w$. Then the probability under $\Q^{\lambda,T}_x$ that no descendant of a particle in $\mathcal{W}$ is sampled by the rate-$\lambda$ Poisson process on $\mathcal{Z}(T)$ is given by
\begin{align*}
\Q_x^{\lambda,T}\left( \mu (\mathcal{W} ) = 0 | \Leb(\mathcal{W}) = w \right) = P_w \left[ \prod_{i=1}^d \exp( - \lambda_i Z_i(T) ) \right] = \exp \left( - \langle w, u(T,\lambda) \rangle \right),
\end{align*}
where we used \eqref{branch exp 2} to obtain the final equality above. 
Equation \eqref{r eq} follows immediately by letting $w = x \mathbf{e}_i$ for some $x > 0$. 
\ep

\bp[Proof 2: Taylor's expansion]
The $i^{\text{th}}$ component of the Laplace exponent $u (T, \cdot )$ is an analytic function mapping $\mathbb{R}^d $ to $\mathbb{R}$. Taylor expanding $u(T, 0)$ about the point $\lambda$ in $\mathbb{R}^d$, for $y \in \mathbb{R}^n$ we have
\begin{align*}
u(T,\lambda + y  )_i = u(T,\lambda)_i + \sum_{ \alpha \neq 0 \in \mathbb{Z}^d_{\geq 0}  } \frac{ y^\alpha}{ \alpha!} D^\alpha u(T,\lambda)_i. 
\end{align*}
Now set $y = - \lambda$, use the fact that $u(T,0) = 0$, and compare with \eqref{outdegree rates}. 
\ep

The following Proposition is a summary of our work in this section. It is an immediate consequence of replacing $0$ with $s$, $T$ with $t$ and $\lambda$ with $\xi$, as well as using the branching property and the Markov property to combine Corollary \ref{cor:poisson} and Lemma \ref{taylor}.

\begin{prop} \label{prop:summary}
Let $s < t$ be times, and let $\xi$ be an element of $\mathbb{R}_{ \geq 0}^d$. Take a $d$-dimensional continuous state branching process with Laplace exponent $u(t,\xi)$ and run a $\xi$-sample $\textbf{b} = (b_{i,j} : 1 \leq i \leq d, 1 \leq j \leq \mathcal{S}_i )$ on the time $t$ population. We say a particle $a$ in the time $s$ population has outdegree $\alpha$ if for each $i$, $a$ is the ancestor of precisely $\alpha_i$ of the particles $\{ b_{i,1},\ldots,b_{i,\mathcal{S}_i } \}$. Then conditional on $\mathcal{F}_s$, 
\begin{itemize}
\item Particles of type $i$ with outdegree $\alpha$ occur within the time $s$ population according to a Poisson process of rate
\begin{align*}
r_i^\alpha(t-s,\xi) := \frac{ (-1)^{|\alpha| + 1 } \xi^\alpha }{ \alpha! } D^\alpha u_i(t-s,\xi).
\end{align*}
\item Particles of type $i$ with nonzero outdegree equal to any $\alpha \in \mathbb{Z}_{ \geq 0}^d - \{0\}$ occur within the time $s$ population according to a Poisson process of rate $u_i(t-s,\xi)$. 
\end{itemize}
\end{prop}

Recall that $\left( \mathcal{F}_t \right)_{ 0 \leq t \leq T}$ is the filtration that measures the CSBP and its genealogy, and that $\tilde{\mathcal{F}}_T$ is an extension of the filtration $\mathcal{F}_T$ measuring the random variable $b = (b_{i,j} : 1 \leq i \leq d, 1 \leq j \leq \mathcal{S}_i )$. 

For an element $u$ of $\mathcal{Z}(t)$, let $I(u)$ denote the $\tilde{\mathcal{F}}_T$-measurable  random variable
\begin{align*}
I(u) = 
\begin{cases}
1 \qquad &\text{if $u$ is an ancestor to some $b_{i,j}$},\\
0 \qquad &\text{otherwise.}
\end{cases}
\end{align*}
We say an element $u$ of $\mathcal{Z}(t)$ is \emph{marked} if $I(u) = 1$. We now define the extension $\mathcal{F}^*_t$ of $\mathcal{F}_t$ by setting
\begin{align*}
\mathcal{F}^*_t = \sigma \left( \mathcal{F}_t , \{ I(u) : u \in \mathcal{Z}(t) \} \right) .
\end{align*} 
That is $\mathcal{F}^*_t$ is the $\sigma$-algebra measuring the CSBP and its genealogy at time $t$, that also knows which time $t$ particles are ancestors to members of the sample $\textbf{b}$. 
Now given a particle $u$ in $\mathcal{Z}(s)$, define the random variable $\mu_{s,t}(u)$ to be the multi-index whose $j^{\text{th}}$ component counts the number of type $j$ marked descendants of $u$ in the time $t$ population. In other words, the $j^{\text{th}}$ component is given by 
\begin{align*}
\mu_{s,t}( u)_j :=  \# \{ v \in \mathcal{Z}_j(t) : I(v) = 1, u \leq v \}.
\end{align*}

\begin{lem} \label{lem:GW}
Under $\mathbb{Q}_x^{\lambda,T}$, define the stochastic process $\left( N(t) \right)_{0 \leq t \leq T}$ taking values in $\mathbb{Z}_{ \geq 0}^d$ by letting 
\begin{align*}
N_i(t) = \# \{ u \in \mathcal{Z}_i (s) : I(u)  = 1 \}.
\end{align*}
Then under the filtration $\left( \mathcal{F}^*_t \right)_{ 0 \leq t \leq T}$, the process $\left( N(t) \right)_{0 \leq t \leq T}$ is a time inhomogenous multidimensional continuous-time Galton-Watson process with the following probabilities. The process starts with a random number of individuals with law 
\begin{align} \label{eq:starting}
\mathbb{P} \left( N(0) = n \right) = \prod_{ i =1}^d \frac{ ( x_i u_i(T,\lambda))^{n_i} e^{ - x_i u_i(T,\lambda)} }{ n_i!}.
\end{align}
Now suppose we have a marked individual $u$ of type $i$ at time $s$. Then the probability that for each $1 \leq j \leq d$ this individual has $\alpha_j$ descendants of type $j$ at a later time $t$ is given by
\begin{align} \label{eq:transition}
\mathbb{Q}_x^{ \lambda,T} \left( \mu_{s,t}(u) = \alpha | \mathcal{F}^*_s, u \in \mathcal{Z}_i(s), I(u) = 1 \right) = \frac{ r_i^\alpha  \left( t -s , u(T-t, \lambda)   \right) }{ u_i(T-t,\lambda)}.
\end{align}
\end{lem}
\begin{proof}
Consider the set $\{ u \in \mathcal{Z}(s) : I(u) = 1 \}$ of marked particles at time $s$. Using the branching property in conjunction with the independent increments of Poisson processes, for any $t > s$ the random variables
\begin{align*}
 \left\{ \mu_{s,t}( u) : u \in \mathcal{Z}(s) : I(u) = 1 \right\}
\end{align*}
are independent. In other words, the process $\left( N(t) \right)_{ 0 \leq t \leq T}$ is a continuous-time Galton-Watson process. 

The equation \eqref{eq:starting} for the starting probabilities is an immediate consequence of Lemma \ref{taylor} --- indeed, marked particles occur in the type $i$ time $0$ population $\mathcal{Z}_i(0)$ according to a Poisson process of rate $u_i(T,\lambda)$. 

We now turn to proving \eqref{eq:transition}. On the one hand, we observe that by Lemma \ref{taylor}, marked particles occur within the time $t$ type $i$ population according to a Poisson process of rate $u_i(T-t,\lambda)$. In particular, setting $\xi = u(T-t,\lambda)$ in Proposition \ref{prop:summary}, we see that particles of type $i$ satisfying $\mu_{s,t}( u) = \alpha$ occur within the time $s$ population according to a Poisson process of rate $r_i^\alpha(t-s,u(T-t,\lambda))$. 

On the other hand, particles of type $i$ satisfying $I(u) =1$ occur within the time $s$ population according to a Poisson process of rate $u_i(T,\lambda)$.

In particular, by dividing $r_i^\alpha(t-s,u(T-t,\lambda))$ through by $u_i(T-t,\lambda)$, we obtain the $\mathcal{F}_s^*$-conditional probability that a marked particle of type $i$ at time $s$ has $\mu_{s,t}(u) = \alpha$, which is precisely \eqref{eq:transition}.
\end{proof}

We are now equipped to understand the law of the ancestral forests under $\mathbb{Q}_x^{\lambda,T}$ up to a combinatorial factor that depends on the forest $\mathcal{H}$ but is independent of $x$ and $\lambda$. Recall from Section \ref{sec:ancestral} that the energy of the forest $\mathcal{H}$ in $\mathbb{H}_k(m)$ over a mesh $\mathsf{t}$ is given by 
\begin{align} \label{eq:energy2}
\mathsf{E}_{\mathsf{t},x} \left( \mathcal{H} , \lambda \right) := (-1)^{k - \rho} x^\rho \prod_{ \ell = 0}^{m-1} \prod_{ v \in V^{(\ell)} } D_{\tau(v)}^{\mu(v)}[ u] \left( \Delta t_\ell , u( T-t_{\ell+1} , \lambda) \right).
\end{align}

\begin{cor} \label{cor:tree}
Let $\mathsf{t}  = (t_0, t_1,\ldots,t_m)$ be a mesh of $[0,T]$. Then there is a combinatorial factor $C_\mathcal{H}$ depending only on the forest $\mathcal{H}$ such that 
\begin{align*}
\mathbb{Q}_x^{\lambda,T} \left( \mathrm{For}_{\mathsf{t}}(\textbf{b}) = \mathcal{H} \right) =C_\mathcal{H}\lambda^k e^{ - \langle x, u(T,\lambda) \rangle}\mathsf{E}_{\mathsf{t},x} \left( \mathcal{H} , \lambda \right) .
\end{align*} 
\end{cor}

\begin{proof}
Using the fact that the marked particles form a continuous-time Galton-Watson tree to obtain the first equality below, we see that there is a combinatorial factor $C'_\mathcal{H}$ such that
\begin{align} \label{eq:gaspard}
&\mathbb{Q}_x^{\lambda,T} \left( \mathrm{For}_{\mathsf{t}}(\textbf{b}) = \mathcal{H} \right) \nonumber \\
&=C'_\mathcal{H}\mathbb{P}\left( N(0) = \rho \right) \prod_{ \ell = 0}^{m-1} \prod_{ v \in V^{(\ell)} }\mathbb{Q}_x^{ \lambda,T} \left( \mu_{t_\ell,t_{\ell+1}}(u) = \mu(v) \big|  \mathcal{F}^*_{t_\ell} , I(u) = 1, \text{$u$ is of type $\tau(v)$} \right).
\end{align}
Now on the one hand, by \eqref{eq:starting}, we have
\begin{align} \label{eq:pogorelich}
\mathbb{P} \left( N(0) = \rho \right) = \frac{ \left( x u(T,\lambda) \right)^\rho e^{ - \langle x, u(T,\lambda) \rangle } }{ \rho!} .
\end{align}
On the other hand, using the formula \eqref{eq:transition} for the transition rates, we have 
\begin{align} \label{eq:sokolov}
\mathbb{Q}_x^{ \lambda,T} \left( u_{t_\ell,t_{\ell+1}}(u) = \mu(v) \big|  \mathcal{F}^*_{t_\ell} , I(u) = 1, \text{$u$ is of type $\tau(v)$} \right)
&=\frac{ r_{\tau(v)}^{ \mu(v)}  \left( \Delta t_\ell  , u(T-t_{\ell+1}, \lambda)   \right) }{ u_{\tau(v)} (T-t_\ell,\lambda)}.
\end{align}
Using equations \eqref{eq:pogorelich} and \eqref{eq:sokolov} in \eqref{eq:gaspard}, we obtain
\begin{align*}
\mathbb{Q}_x^{\lambda,T} \left( \mathrm{For}_{\mathsf{t}}(\textbf{b}) = \mathcal{H} \right) =C'_\mathcal{H}\frac{ x^\rho  u(T,\lambda)^\rho e^{ - \langle x, u(T,\lambda) \rangle } }{ \rho!}  \prod_{ \ell = 0}^{m-1} \prod_{ v \in V^{(\ell)} }\frac{ r_i^\alpha   \left( \Delta t_\ell  , u(T-t_{\ell+1}, \lambda)   \right) }{ u_{\tau(v)} (T-t_\ell,\lambda)}. 
\end{align*}
Now using the definition \eqref{outdegree rates} of $r_i^\alpha(t,\lambda)$ as well as the definition of forest energy \eqref{eq:energy2}, for a different combinatorial factor $C_\mathcal{H}$ we have 
\begin{align*}
\mathbb{Q}_x^{\lambda,T} \left( \mathrm{For}_{\mathsf{t}}(\textbf{b}) = \mathcal{H} \right) =C_\mathcal{H} e^{ - \langle x, u(T,\lambda) \rangle } \mathsf{E}_{\mathsf{t},x} \left( \mathcal{H} , \lambda \right)  \left\{ u(T,\lambda)^\rho   \prod_{ \ell = 0}^{m-1} \prod_{v \in V^{(\ell) } } \frac{ u (T-t_{\ell+1} , \lambda)^{ \mu(v) }   }{ u_{\tau(v)} (T - t_\ell, \lambda)   }   \right\}.
\end{align*}
We now note that we have a cascade of cancellations in the final term above. Indeed since $  \prod_{ \ell = 0}^{m-1} \prod_{v \in V^{(\ell) } }   u (T-t_{\ell+1} , \lambda)^{ \mu(v) }   = \prod_{ \ell = 1}^m \prod_{v \in V^{(\ell)} }  u_{\tau(v)}(T-t_\ell , \lambda)$ we have 
\begin{align*}
 u(T,\lambda)^\rho   \prod_{ \ell = 0}^{m-1} \prod_{v \in V^{(\ell) } } \frac{ u (T-t_{\ell+1} , \lambda)^{ \mu(v) }   }{ u_{\tau(v)} (T - t_\ell, \lambda)   }   = u(0,\lambda)^k = \lambda^k,
\end{align*}
proving the result.
\end{proof}

We are now equipped to prove Theorem \ref{thm:CSBP1}, which is a consequence of Corollary \eqref{cor:tree} once we have established that
\begin{align*}
C_\mathcal{H} = 1/k! \text{ for every $\mathcal{H}$ in $\mathbb{H}^k(m)$} .
\end{align*}
It is possible to obtain the combinatorial factor $C_\mathcal{H}$ in two different ways. One approach, the more direct one, is to use a direct counting argument involving taking a random labelling of the leaves of a Galton-Watson tree and studying the probability that the resulting labelled tree is equal to some particular $\mathcal{H} $ in $\mathbb{H}^k(m)$. We opt for an alternative approach using the recent extension of Fa\`a di Bruno's formula that we hope will clarify several structural aspects of our formulas. Indeed, adapting Theorem 2.1 of Johnston and Prochno \cite{JP} by using the following Remark 2.2 of that paper, we see that if $f:\mathbb{R}^d \to \mathbb{R}$ and $F^{(1)},\ldots,F^{(m-1)}: \mathbb{R}^d \to \mathbb{R}^d$ are smooth functions, then for a multi-index $k$ their $k^{\text{th}}$ derivative satisfies
\begin{align} \label{eq:fdb2}
D^k \left[ f \circ F^{(1)} \circ \ldots \circ F^{(m)} \right] (\lambda) = \sum_{ \mathcal{H} \in \mathbb{H}^k(m) } D^\rho[f] \circ F^{[0,m-1]} (\lambda)  \prod_{ \ell = 0}^{m-1} \prod_{ v \in V^{(\ell)} } D_{\tau(v)}^{\mu(v) } [ F^{(\ell) } ] \circ F^{[ l+1,m]}
\end{align}
where $F^{[i,j]} := F^{(i)} \circ \ldots \circ F^{(j)}$ if $j \geq i$, and is equal to the identity map otherwise. Now consider writing the map $\lambda \mapsto e^{ -  \langle x, u(T,\lambda) \rangle }$ in terms of the chain composition 
\begin{align*}
e^{ - \langle x , u(T,\lambda) \rangle } = f_x \circ F^{(1)} \circ \ldots \circ F^{(m-1)},
\end{align*}
where $f_x(\lambda) = e^{ - \langle x, \lambda \rangle }$ and $F^{(\ell)} (\lambda ) := u( \Delta t_\ell , \lambda)$. It follows by \eqref{eq:fdb2} and \eqref{eq:energy2} that 
\begin{align} \label{eq:partition}
(-1)^k D^k e^{ - \langle x , u(T,\lambda)  \rangle} = e^{ - \langle x, u(T,\lambda) \rangle}  \sum_{ \mathcal{H} \in \mathbb{H}^k(m) }  \mathsf{E}_{x,\mathsf{t}} \left( \mathcal{H} , \lambda \right). 
\end{align}
On the other hand, we note that by definition, we may write $\{ \S = k\}$ in terms of the disjoint union
\begin{align} \label{eq:disjoint}
\{ \S = k \} = \bigcup_{ \mathcal{H} \in \mathbb{H}^k(m)  } \{ \mathrm{For}_{\mathsf{t}}(\textbf{b})= \mathcal{H} \}.
\end{align}
In particular, combining \eqref{eq:disjoint} with Corollary \ref{cor:tree}, on the one hand we obtain
\begin{align} \label{sibelius}
\Q_x^{ \lambda, T} \left( \S = k \right) = \lambda^k e^{ - \langle x, u(T,\lambda) \rangle} \sum_{ \mathcal{H} \in \mathbb{H}^k(m)  }C_\mathcal{H}  \mathsf{E}_{x,\mathsf{t}} \left( \mathcal{H} , \lambda \right). 
\end{align}
On the other hand, a direct computation using Poisson processes and interchanging differentiation and expectation (which is justified by Lemma \ref{techlem}), we have
\begin{align} \label{eq:mahler}
\Q_x^{ \lambda, T} \left( \S = k \right) &= P_x  \frac{ \lambda^kZ(T)^k }{ k!}  e^{ - \langle \lambda, Z(T) \rangle} \nonumber \\
&= \frac{ (-1)^k \lambda^k }{ k!} D^k e^{ - \langle x, u(T,\lambda) \rangle } \nonumber \\
&= \frac{ \lambda^k}{k!} e^{ - \langle x, u(T,\lambda) \rangle}  \sum_{ \mathcal{H} \in \mathbb{H}^k(m) }  \mathsf{E}_{x,\mathsf{t}} \left( \mathcal{H} , \lambda \right),
\end{align}
where the final equality above follows from \eqref{eq:partition}. In order for \eqref{eq:mahler} and \eqref{sibelius} to hold for all $\lambda, \mathsf{t}$, it must be the case that $C_{\mathcal{H}} = 1/k!$, thereby completing the proof of Theorem \ref{thm:CSBP1}.

\section{Proofs of the local structure of CSBP coalescence} \label{sec:local proof}

\subsection{Small time asymptotics of CSBP genealogy and $\Lambda$-coalescents} \label{sec:smalltime}
Recall that for a forest $\mathcal{H}$ in $\mathbb{H}^k(1)$, the energy across a mesh $(0,t)$ is given by 
\begin{align} \label{eq:special2}
\mathsf{E}_{(0,t),x} ( \mathcal{H} , \lambda ) := (-1)^{k- \rho} x^\rho \prod_{ v \in V^{(0)} } D_{ \tau(v) }^{\mu(v)}[u] (t, \lambda), \qquad \text{$\mathcal{H}$ in $\mathbb{H}^k(1)$}.
\end{align}
(This is the special case $m=1$ with $t_1 = t$ of the formula \eqref{eq:energy2} for the energy of the tree.) Recall from Definition \ref{def:special} the forest $\mathcal{H}^k_{c,\alpha}$ of $\mathbb{H}^k(1)$ in which all but one tree is a stick (a stick being a tree with a single leaf, in which the root has the same type as the leaf), and the single non-stick tree has root type $c$ and leaf types $\alpha$. We note that the energy of $\mathcal{H}^k_{c,\alpha}$ is given by
\begin{align} \label{eq:special3}
\mathsf{E}_{(0,t),x} ( \mathcal{H}^k_{c,\alpha}  , \lambda ) := (-1)^{|\alpha| - 1}  x^{k- \alpha + \mathbf{e}_c} D_c^\alpha [u](t,\lambda) \prod_{j=1}^d D_j^{\mathbf{e}_j}[u](t,\lambda)^{k_j - \alpha_j}.
\end{align}
In particular, plugging \eqref{eq:special3} into Theorem \ref{thm:CSBP2}, we have  
\begin{align} \label{eq:special4}
&\mathbb{P}_x^{k,t} \left( \mathrm{For}_{(0,t)}(\textbf{b}) = \mathcal{H}^k_{c,\alpha} , \mathcal{Z}(t) \succ k  \right) \nonumber \\
&=  (-1)^{|\alpha| - 1} \frac{x^{k- \alpha + \mathbf{e}_c}}{k!}  \int_{\mathbb{R}_{ \geq 0}^d } \pi^k( \mathrm{d} \lambda) \lambda^k e^{ - \langle x, u(t,\lambda) \rangle}  D_c^\alpha [u](t,\lambda) \prod_{j=1}^d D_j^{\mathbf{e}_j}[u](t,\lambda)^{k_j - \alpha_j}.
\end{align}
The following lemma gives a first characterisation of the small-$t$ asymptotics of the integral in \eqref{eq:special4} in terms of the branching mechanism. 

\begin{lem} \label{lem:swap}
We have the first integral equation 
\begin{align*}
\lim_{t \to 0} \frac{1}{t} \mathbb{P}_x^{k,t} \left( \mathrm{For}_{(0,t)}(\textbf{b}) = \mathcal{H}^k_{c,\alpha} , \mathcal{Z}(t) \succ k \right) =  (-1)^{|\alpha|}  \frac{ x^{ k - \alpha + \mathbf{e}_c} }{k!}  \int_{ \mathbb{R}_{ \geq 0}^d } \pi^k( \mathrm{d} \lambda)  \lambda^k  e^{ - \langle x, \lambda \rangle } D_c^\alpha [ \psi]( \lambda).
\end{align*}
\end{lem}
\begin{proof}
Letting $H(t) := \mathbb{P}_x^{k,t} \left( \mathrm{For}_{(0,t)}(\textbf{b}) = \mathcal{H}^k_{c,\alpha}    \right)$, since $H(0) = 0$ we would like to compute $\frac{ \partial}{ \partial t} H(t) \big|_{t = 0}$. Now using the definition of $\pi^k( \mathrm{d} \lambda)$ (see Theorem \ref{thm:po c} and the preceding equation \eqref{solo def}) we have
\begin{align*}
\pi^k( \mathrm{d} \lambda) \lambda^k = C_k \prod_{ i : k_i = 0 } \delta_0(  \mathrm{d} \lambda_i)  \prod_{ j : k_j > 0 } \lambda^{k_j - 1}   \mathrm{d} \lambda_j  
\end{align*}
where $C_k = \prod_{ i : k_i > 0} k_i$. In particular, by \eqref{eq:special4} we have 
\begin{align*}
H(t) = (-1)^{|\alpha|-1}  C_k \frac{x^{ k - \alpha + \mathbf{e}_c}}{k!}  \int_{\mathbb{R}_{ \geq 0}^d }  \prod_{ i : k_i = 0 } \delta_0(  \mathrm{d} \lambda_i)  \prod_{ j : k_j > 0 } \mathrm{d} \lambda_j   \lambda_j^{k_j - 1}    H(t,\lambda),
\end{align*}
where $H(t,\lambda)$ is given by 
\begin{align} \label{eq:H}
H(t,\lambda) :=  e^{ - \langle x, u(t,\lambda) \rangle }D_c^\alpha [u](t,\lambda) \prod_{j=1}^d D_j^{\mathbf{e}_j}[u](t,\lambda)^{k_j - \alpha_j}.
\end{align} 
The integrand $H(t,\lambda)$ is continuously differentiable in $(t,\lambda)$ in the set $\mathbb{R}_{ \geq 0} \times \{ \lambda \in \mathbb{R}_{ \geq 0}^d : \lambda \prec k \}$ (where, as above, $\lambda \prec k$ means $\lambda_i = 0$ whenever $k_i = 0$). In particular, we may differentiate through the integral
\begin{align*}
\frac{ \partial}{ \partial t} H(t) =  (-1)^{|\alpha|-1}  C_k \frac{x^{ k - \alpha + \mathbf{e}_c}}{k!}  \int_{\mathbb{R}_{ \geq 0}^d }  \prod_{ i : k_i = 0 } \delta_0(  \mathrm{d} \lambda_i)  \prod_{ j : k_j > 0 } \mathrm{d} \lambda_j   \lambda_j^{k_j - 1} \frac{ \partial}{ \partial t} H(t,\lambda) 
\end{align*}
The proof follows after we establish the following equation:
\begin{align} \label{eq:zero}
\frac{ \partial}{ \partial t } H(\lambda, t) \big|_{t = 0 } = e^{ - \langle x ,\lambda \rangle } D_c^\alpha [ \psi] ( \lambda). 
\end{align}
To prove \eqref{eq:zero}, recall from \eqref{csbp pde multi} that $u(t,\lambda)$ is smooth in $(t,\lambda)$ and solves the partial differential equation
\begin{align} \label{eq:q1}
\frac{ \partial}{ \partial t} u(t,\lambda) = - \psi \left( u(t,\lambda) \right) \qquad u(0,\lambda) = \lambda,
\end{align} 
so that in particular, 
\begin{align} \label{eq:q2}
D_i^{\mathbf{e}_i} [ u ] (t,\lambda) \Bigg|_{t = 0 } = 1 \qquad \text{and} \qquad D_c^\alpha [ u] (t,\lambda) \Bigg|_{t = 0 } = 0 \text{ when $\alpha \neq \mathbf{e}_c$}.
\end{align}
Now by interchanging differential operators $\frac{ \partial}{ \partial t}$ and $D^\alpha$, we have 
\begin{align} \label{eq:q4} 
\frac{ \partial}{ \partial t} D^\alpha_c [u](t,\lambda)  \Bigg|_{t = 0 }= - D^\alpha_c [ \psi] (\lambda).
\end{align}
Now differentiate $H(\lambda,t)$ in \eqref{eq:H} using the product rule. Appealing to \eqref{eq:q1}, \eqref{eq:q2} and \eqref{eq:q4}, we obtain \eqref{eq:zero}, completing the proof of Lemma \ref{lem:swap}.
\end{proof}

We are now ready to prove Theorem \ref{thm:rate}, which expresses the integral in Lemma \ref{lem:swap} in terms of an integral against a measure on $[0,1]^d$ depending on the branching mechanism. Recall that the  $c^{\text{th}}$ component of the branching mechanism is given by 
\begin{align} \label{eq:lkrep2}
\psi_c(\la) = - \sum_{ j = 1}^d \kappa_{c,j} \lambda_j + \beta_c \lambda_c^2 + \int_{\mathbb{R}_{\geq 0}^d} \left( e^{ - \langle \la , r \rangle } - 1 +  \lambda_i r_i \ind_{r_i \leq 1} \right) \nu_c(\mathrm{d}r).
\end{align}                                                                                                                                                                                                                                                                                                                                                                                                                                                                                                                                                                                                                                                                                                                                                                                                                                                                                                                                                                                                                                                                                                                                                                                                                                                                                                                                                                                                                                                                                                                                                                                                                                                                                                                                                                                                                                                                                                                                                                                                                                                                                                                                                                                                                                                                                                                                                                                                                                                                                                                                                                                                                                                                                                                                                                                                                                                                                                                                                                                                                                                                                                                                                                                                                                                                                                                                                                                                                                                                                                                                                                                                                                                                                                                                                                                                                                                                                                                                      
We also recall the associated measure on $[0,1)^d$ defined by
\begin{align} \label{pitman measure}
\Lambda^\psi_c(x,\mathrm{d}s) := \frac{ 2 \beta_c }{ x_c } \delta_0 +x_c s_c^2 T_x^\# \nu_c  ( \mathrm{d}s),
\end{align} 
where $T_x^{ \# }\nu_c$ is the pushforward of $\nu$ under the map $T_x:\mathbb{R}_{ \geq 0}^d \to [0,1)^d$ with components $T_x(r)_i := \frac{ r_i}{ x_i + r_i}$. 

Finally, in the proof of Theorem \ref{thm:rate} we will require the following property of integration against $ \pi^k(  \mathrm{d} \lambda) $, which is easily proved by using the definition of the gamma integral. 
\begin{align} \label{eq:gamma fact}
 \int_{ \mathbb{R}_{\geq 0}^d } \pi^k(  \mathrm{d} \lambda) \lambda^k e^{ - \langle y, \lambda \rangle } = \frac{k!}{y^k} \qquad \text{for $y \in \mathbb{R}_{ > 0}^d$.}
\end{align}

\bp[Proof of Theorem \ref{thm:rate}]
Let $c \in \{1,\ldots,d\}$ and let $\alpha \neq 0,\mathbf{e}_c$ be a multi-index. By virtue of Lemma \ref{techlem} we may differentiate through the integral in the L\'evy-Khintchine representation \eqref{eq:lkrep2}, so that 
\begin{align} \label{big alpha}
D^{\alpha}_c [ \psi ] (\lambda) =  - \sum_{ j = 1}^d \kappa_{c,j} \ind_{ \alpha = \mathbf{e}_j } + 2 \beta_c \ind_{\alpha = 2 \mathbf{e}_c} + (-1)^{|\alpha|} \int_{\mathbb{R}_{ \geq 0}^d} r^\alpha e^{ - \langle \lambda , r \rangle} \nu_c( \mathrm{d} r). 
\end{align}
Using \eqref{big alpha} and Lemma \ref{lem:swap} to obtain the first equality below, and Fubini's theorem to obtain the second, we have 
\begin{align*}
&\lim_{t \to 0} \frac{1}{t} \mathbb{P}_x^{k,t} \left( \mathrm{For}_{(0,t)}(\textbf{b}) = \mathcal{H}^k_{c,\alpha} , \mathcal{Z}(t) \succ k  \right) \\ 
&= \frac{ x^{k- \alpha + \mathbf{e}_c } }{k!} \int_{ \mathbb{R}_{ \geq 0 }^d } \pi^k (  \mathrm{d} \lambda) \lambda^k e^{ - \langle \lambda, x \rangle }  \left( - \sum_{ j = 1}^d \kappa_{c,j} \ind_{ \alpha = \mathbf{e}_j } + 2 \beta_c \ind_{\alpha = 2 \mathbf{e}_c} + \int_{ \mathbb{R}_{ \geq 0 }^d } r^\alpha e^{ - \langle \lambda , r \rangle } \nu_c (\mathrm{d}r) \right)\\
&= \frac{ x^{k- \alpha + \mathbf{e}_c } }{k!} \left( - \sum_{ j = 1}^d \kappa_{c,j} \ind_{ \alpha = \mathbf{e}_j } + 2 \beta_c \ind_{\alpha = 2 \mathbf{e}_c} \right)   \int_{ \mathbb{R}_{ \geq 0 }^d } \pi^k (  \mathrm{d} \lambda) \lambda^k e^{ - \langle \lambda, x \rangle }\\
& +                \frac{ x^{k- \alpha + \mathbf{e}_c } }{k!} \int_{ \mathbb{R}_{ \geq 0}^d } \nu_c( \mathrm{d} r) r^\alpha \int_{\mathbb{R}_{ \geq 0}^d} \pi^k ( \mathrm{d} \lambda) \lambda^k e^{ - \langle x +r, \lambda \rangle } .
\end{align*}  
Using the gamma integral identity \eqref{eq:gamma fact}, this simplifies to 
\begin{align*}
&\lim_{t \to 0} \frac{1}{t} \mathbb{P}_x^{k,t} \left( \mathrm{For}_{(0,t)}(\textbf{b}) = \mathcal{H}^k_{c,\alpha} , \mathcal{Z}(t) \succ k   \right) \\
&= \frac{ x^{k- \alpha + \mathbf{e}_c } }{k!} \left( - \sum_{ j = 1}^d \kappa_{c,j} \ind_{ \alpha = \mathbf{e}_j } + 2 \beta_c \ind_{\alpha = 2 \mathbf{e}_c} \right)  \frac{k!}{x^k}  +      \frac{ x^{k- \alpha + \mathbf{e}_c } }{k!} \int_{ \mathbb{R}_{ \geq 0}^d } \nu_c( \mathrm{d} r) r^\alpha \frac{k!}{ (x+r)^k} \\
&=- \sum_{ j = 1}^d \kappa_{c,j} \ind_{ \alpha = \mathbf{e}_j } \frac{x_c}{x_j} + 2 \beta_c \ind_{\alpha = 2 \mathbf{e}_c} \frac{1}{x_c} + x_c \int_{ \mathbb{R}_{ \geq 0}^d } \nu_c( \mathrm{d} r)  \left( \frac{ r }{ x+ r } \right)^\alpha \left( \frac{ x}{ x+r} \right)^{k-\alpha}\\
&=- \sum_{ j = 1}^d \kappa_{c,j} \ind_{ \alpha = \mathbf{e}_j } \frac{x_c}{x_j} + 2 \beta_c \ind_{\alpha = 2 \mathbf{e}_c} \frac{1}{x_c} + x_c \int_{ [0,1]^d }T_x^{ \#} \nu_c( \mathrm{d} s) s^\alpha (1 -s)^{ k -\alpha},
\end{align*}  
where the final equality above follows simply from the definition of the pushforward measure $T_x^{\#} \nu_c$. That completes the proof of Theorem \ref{thm:rate}. 
\ep

\section*{Acknowledgements}
The first author was supported in the first part by the ERC grant \emph{Integrable Random Structures} and in the latter part by the FWF grant \emph{Asymptotic Geometric Analysis and Applications}.


\end{document}